\RequirePackage{fix-cm}
\documentclass[smallextended]{svjour3}       
\smartqed  
\usepackage{graphicx}
\usepackage{amsmath,amssymb}
\usepackage{amsfonts}
\usepackage{bbm, dsfont}
\usepackage{enumerate}
\usepackage[active]{srcltx}

\usepackage{amscd}

\journalname{Ricerche di Matematica}

\def\e{{\varepsilon}}

\numberwithin{lem}{section}

\newcommand{\ds}{\displaystyle}

\allowdisplaybreaks

\begin{document}

\title{On attainability of optimal controls in coefficients for system of Hammerstein type with anisotropic p-Laplacian}

\titlerunning{Optimal controls for Hammerstein system with anisotropic p-Laplacian}

\author{Tiziana Durante \and Olha P. Kupenko\and Rosanna Manzo}

\authorrunning{T. Durante, O.P. Kupenko, R. Manzo}

\institute{T. Durante \at
            Universit\`{a} degli Studi di Salerno, Dipartimento di Ingegneria dell'Informazione \\
           ed Elettrica e Matematica Applicata, \\
           Via Giovanni Paolo II, 132, 84084 Fisciano (SA), Italy\\
            \email{tdurante@unisa.it}
                  \and
          O. Kupenko \at
          National Mining University, Department of System Analysis and Control, \\
          Yavornitskyi av., 19, 49005 Dnipro, Ukraine, \\
          National Technical University of Ukraine ``Kiev Polytechnical Institute'', \\
          Institute for Applied and System Analysis, \\
          Peremogy av., 37, build. 35, 03056 Kiev, Ukraine\\
          \email{kogut\_olga@bk.ru}
            \and
            R. Manzo \at
            Universit\`{a} degli Studi di Salerno, Dipartimento di Ingegneria dell'Informazione \\
           ed Elettrica e Matematica Applicata, \\
           Via Giovanni Paolo II, 132, 84084 Fisciano (SA), Italy\\
            \email{rmanzo@unisa.it}
}

\date{Received: date / Accepted: date}

\maketitle

\begin{abstract}
In this paper we consider an optimal control
problem (OCP) for the coupled system of a nonlinear monotone Dirichlet problem with anisotropic $p$-Laplacian and matrix-valued $L^\infty(\Omega,\mathbb{R}^{N\times N})$-controls in its coefficients
  and a nonlinear equation of Hammerstein type. Using the direct
method in calculus of variations, we prove the existence of an optimal control in considered problem and provide sensitivity analysis for a specific case of considered problem with respect to two-parameter regularization.
\end{abstract}

\section{Introduction}

The aim of this paper is  to prove the
existence result for an optimal control problem (OCP) governed by the system of a homogeneous Dirichlet nonlinear  elliptic boundary value problem, whose principle part is an ani\-so\-tro\-pic $p$-Laplace-like operator, and a nonlinear equation of Hammerstein type, and to provide sensitivity analysis for the specific case of considered optimization problem with respect to a two-parameter regularization.
As controls we consider the symmetric matrix of anisotropy in the main part of the elliptic equation. We assume that admissible controls are measurable and uniformly bounded matrices of $L^\infty(\Omega;\mathbb{R}^{N\times N})$.

Systems with distributed parameters and optimal control problems for systems described by PDE, nonlinear integral and ordinary differential equations
have been widely studied by many authors (see for example \cite{IvanMel,KuMa15,Lions0,Lurie,Zgurovski99}). However, systems which contain equations of different types and optimization problems associated with them are still less well understood. In general case including as well  control and state constraints, such problems are rather complex and have no simple constructive solutions. The system, considered in the present paper, contains two equations: a nonlinear elliptic equation with the so-called anisotropic $p$-Laplace operator with homogeneous Dirichlet boundary conditions and a nonlinear equation of Hammerstein type, which nonlinearly depends on the solution of the first object.  The optimal control problem we study here is to minimize the discrepancy between a given distribution $z_d\in L^p(\Omega)$ and a solution of Hammerstein equation $z=z(A,y)$, choosing an appropriate matrix of coefficients $A\in \mathfrak{A}_{ad}$, i.e.
\begin{equation}
\label{0.1} I(A,y,z)=\int_\Omega |z(x)-z_d(x)|^2\,dx \longrightarrow \inf
\end{equation}
subject to constrains
\begin{gather}
\label{0.2}
z + B F(y,z)=0\quad\mbox{ in }\Omega,\\
\label{0.3}
-\mathrm{div}\big(|(A(x)\nabla y,\nabla y)_{\mathbb{R}^N}|^{(p-2)/2}A(x)\nabla y\big)= f \text{ in }\Omega,\\
A\in \mathfrak{A}_{ad},\quad y=0\text{ on }\partial\Omega,
\label{0.4}
\end{gather}
where $B:L^q(\Omega)\to L^p(\Omega)$ is a positive linear operator, $F:W_0^{1,p}(\Omega)\times L^p(\Omega)\to L^q(\Omega)$ is a nonlinear operator, $f\in L^2(\Omega)$ is a given distribution, and a class of admissible controls $\mathfrak{A}_{ad}$ is a nonempty compact subset of $L^\infty(\Omega;\mathbb{R}^{\frac{N(N+1)}{2}})$.

The interest to equations whose principle part is an ani\-so\-tro\-pic $p$-Laplace-like operator arises from various applied contexts related to composite materials such as nonlinear dielectric composites, whose nonlinear behavior is modeled by the so-called power-low (see, for instance, \cite{BS,LK} and references therein). It is sufficient to say that anisotropic $p$-Laplacian $\Delta_p(A,y)$ has profound background both in the theory of a\-ni\-so\-tro\-pic and nonhomogeneous media and in Finsler or Minkowski geometry \cite{Xia12}. As a rule, the effect of anisotropy appears naturally in a wide class of geometry~--- Finsler geometry. A typical and important example of Finsler geometry is Minkowski geometry. In this case, anisotropic Laplacian is closely related to a convex hypersurface in $\mathbb{R}^N$, which is called the Wulff shape \cite{Xia11}. Since the topology of the Wulff shape essentially depends on the matrix of anisotropy $A(x)$, it is reasonable to take such matrix as a control.  From mathematical point of view, the interest of anisotropic $p$-Laplacian lies on its nonlinearity and an effect of degeneracy, which turns out to be the major difference from the standard Laplacian on $\mathbb{R}^N$.

In practice, the equations of Hammerstein type appear as integral or integ\-ro-differential equations. The class of integral equations is very important for theory and applications, since there are less restrictions on smoothness of the desired solutions involved in comparison to those for the solutions of differential equations. It should be also mentioned here, that well posedness or uniqueness of the solutions is not typical for equations of Hammerstein type or optimization problems associated with such objects (see \cite{AMJA}). Indeed, this property requires rather strong assumptions on operators $B$ and $F$, which is rather restrictive in view of numerous applications (see \cite{VainLav}). The physical motivation of optimal control problems which are similar to those investigated in the present paper is widely discussed in \cite{AMJA,ZMN}.

Using the direct
method of the Calculus of Variations, we show in Section  4 that the optimal control problem \eqref{0.1}--\eqref{0.4} has a nonempty set of solutions provided the admissible controls $A(x)$ are uniformly bounded in $BV$-norm, in spite of the fact that the corresponding quasilinear differential operator $-\mathrm{div} \big(|(A\nabla y,\nabla y)_{\mathbb{R}^N}|^{\frac{p-2}{2}}A\nabla y\big)$, in principle, has degeneracies as $|A^\frac{1}{2}\nabla y|$ tends to zero \cite{Alessand}. Moreover, when the term $|(A\nabla y,\nabla y)_{\mathbb{R}^N}|^{\frac{p-2}{2}}$ is regarded as the coefficient of the Laplace operator, we  have the case of unbounded coefficients (see \cite{HK1,Ko}). In order to avoid degeneracy with respect to the control $A(x)$, we assume that matrix $A(x)$ has a uniformly bounded spectrum away from zero. As for the optimal control problems in coefficients for degenerate elliptic equations and variational inequalities, we can refer to \cite{ButtazzoKogut,CUO09,CUO12,CUPR14,KoLe3,KM2,KuMa15}.

A number of regularizations have been suggested in the literature. See \cite{Roubicek} for a discussion for what has come to be known as $(\e,p)$-Laplace problem, such as $-\mathrm{div}((\e +|\nabla y|^2)^\frac{p-2}{2})\nabla y$. While the $(\e,p)$-Laplacian regularizes the degeneracy as the gradients tend to zero, the term $|\nabla y|^{p-2}$, viewed again as a coefficient, may grow large \cite{CasasFernandez1991}. Therefore, following ideas of \cite{CKL}, for the specific case of considered optimization problem we introduce  yet another regularization that leads to a sequence of monotone and bounded approximation $\mathcal{F}_k(|A^\frac{1}{2}\nabla y|^2)$ of $|A^\frac{1}{2}\nabla y|^2$. As a result, for fixed parameter $p\in [2,\infty)$ and control $A(x)$, we arrive at a two-parameter variational problem governed by operator $-\mathrm{div}((\e +\mathcal{F}_k(|A^\frac{1}{2}\nabla y|^2))^\frac{p-2}{2})A\nabla y$ and a two-parameter Hammerstein equation with non-linear kernel $F_{\e,k}(y,z)=(\e +\mathcal{F}_k(y^2))^\frac{p-2}{2}y+(\e +\mathcal{F}_k(z^2))^\frac{p-2}{2}z$.  Finally, we deal with a two-parameter family of optimal control problems in the coefficients for a system of elliptic boundary value problem and equation of Hammerstein type. We consequently provide the well-posedness analysis for the perturbed optimal control problems in Sections 5. In section 6, we show  that  the solutions
of two-parametric family of perturbed optimal control problems
can be considered as appropriate approximations to optimal pairs for the original problem similar to \eqref{0.1}--\eqref{0.4}. To the end, we note that the approximation and regularization are not only considered to be useful for the mathematical analysis, but also for the purpose of numerical simulations. The numerical analysis as well as the case of degenerating controls are subjects to future publications.
\section{Notation and preliminaries}
\label{Sec 1}
Let $\Omega$ be a bounded open subset of $\mathbb{R}^N$ ($N\ge1$) with a Lipschitz boundary.
Let $p$ be a real number such that $2\le p<\infty$, and let $q=p/(p-1)$ be the conjugate of $p$.
Let $\mathbb{S}^N:=\mathbb{R}^{\frac{N(N+1)}{2}}$ be the set of all
symmetric matrices $A=[a_{ij}]_{i,j=1}^N$,
($a_{ij}=a_{ji}\in \mathbb{R}$). We suppose that $\mathbb{S}^N$ is endowed
with the Euclidian scalar product $A\cdot B=
\mathrm{tr}(A\,B)= a_{ij}b_{ij}$  and with the
corresponding Euclidian norm
$\|A\|_{\mathbb{S}^N}=(A\cdot A)^{1/2}$. We also
make use of the so-called spectral norm $\|A\|_2:=\sup\big\{|A\xi| \ :\ \xi\in \mathbb{R}^N\ \text{ with }\ |\xi| =1\big\}$ of matrices $A\in \mathbb{S}^N$, which is different from the Euclidean norm $\|A\|_{\mathbb{S}^N}$.
However, the relation $\|A\|_2\le \|A\|_{\mathbb{S}^N}\le\sqrt{N} \|A\|_2$ holds true for all $A\in \mathbb{S}^N$.

Let
$L^{1}(\Omega)^{\frac{N(N+1)}{2}}=L^{1}\big(\Omega;\mathbb{S}^N\big)$
be the space of integrable functions whose values are symmetric matrices.
By $BV(\Omega;\mathbb{S}^N)$ we denote the space of all matrices in $L^1(\Omega;\mathbb{S}^N)$ for which the norm
\begin{multline}\label{1.0}
\|A\|_{BV(\Omega;\mathbb{S}^N)} =\|A\|_{L^1(\Omega;\mathbb{S}^N)}+\int_\Omega|D A|= \|A\|_{L^1(\Omega;\mathbb{S}^N)}\\
 + \sum_{1\le i\le j\le N}\sup\Big\{\int_\Omega a_{ij}\,\mathrm{div}\varphi\,dx\ :\ \varphi\in C^1_0(\Omega;\mathbb{R}^N),\ |\varphi(x)|\le 1\ \text{for}\ x\in \Omega\Big\}
\end{multline}
is finite.

\textit{Weak Compactness Criterion in $L^1(\Omega)$.}
Throughout the paper we will often use the concept of  weak and
strong convergence in $L^1(\Omega)$. Let
$\left\{a_\e\right\}_{\e>0}$ be a bounded sequence of functions in
$L^1(\Omega)$. We recall that
$\left\{a_\e\right\}_{\e>0}$ is called equi-integrable on $\Omega$,
if for any $\delta>0$ there is a $\tau=\tau(\delta)$ such that
$\int_S \|a_\e\|\,dx<\delta$ for every measurable
subset $S\subset\Omega$ of Lebesgue measure $|S|<\tau$. Then the
following assertions are equivalent for
$L^1(\Omega)$-bounded sequences:
\begin{enumerate}
\item[(i)] a sequence $\left\{a_k\right\}_{k\in \mathbb{N}}$ is weakly compact in $L^1(\Omega)$;
\item[(ii)] the sequence $\left\{a_k\right\}_{k\in \mathbb{N}}$ is equi-integrable.
\end{enumerate}

\begin{lemma}[Lebesgue's Theorem]
\label{Th_1.9} If a sequence  $\left\{a_k\right\}_{k\in
\mathbb{N}}\subset L^1(\Omega)$ is equi-integ\-rable
and $a_k\rightarrow a$ almost everywhere in $\Omega$ then
$a_k\rightarrow a$ in $L^1(\Omega)$.
\end{lemma}
\begin{lemma}[\cite{Zh2010}]\label{Gikov}
If a sequence $\{\varphi_k\}_{k\in \mathbb{N}}$ is bounded in $L^1(\Omega)$, $\varphi_k\to 0$ a.e. in $\Omega$ and $\{g_k\}_{k\in \mathbb{N}}$ is equi-integrable, then $\varphi_k\cdot g_k\to 0$ strongly in $L^1(\Omega)$.
\end{lemma}
\begin{lemma}[\cite{Zh2010}]\label{Gikov1}
Let $B_n(x,\xi)$ and $B(x,\xi)$ be Caratheodory vector
functions acting from $\Omega\times \mathbb{R}$  to $\mathbb{R}$. These vector
functions are assumed to satisfy the  monotonicity and pointwise convergence conditions
\begin{gather*}
(B_n(x,\xi)-B_n(x,\eta))(\xi-\eta)\ge 0,\quad B_n(x,0)\equiv 0,\\
|B_n(x,\xi)|\le c_0(|\xi|)<\infty;\quad \lim_{n\to \infty} B_n(x,\xi)=B(x,\xi),
\end{gather*}
for a.e. $x\in\Omega$ and every $\xi\in \mathbb{R}$. If $v_n\rightharpoonup v$ in $L^p(\Omega)$, $B_n(x,v_n)\rightharpoonup z$ in $L^q(\Omega)$, then
\begin{equation}
\label{A.1}
\liminf_{n\to\infty} \langle B_n(x,v_n),v_n\rangle_{L^q(\Omega);L^p(\Omega)}\ge \langle z,v\rangle_{L^q(\Omega);L^p(\Omega)},
\end{equation}
and in the case of equality in \eqref{A.1}, we have  $z=B(x,v)$.
\end{lemma}

\textit{Admissible controls.}
Let $\xi_1$, $\xi_2$ be given elements of $L^\infty(\Omega)\cap BV(\Omega)$ satisfying the conditions
\begin{gather}\label{1.1}
0<\alpha\le\xi_1(x)\le\xi_2(x)\text{ a.e. in }\Omega,
\end{gather}
where $\alpha$ is a given positive value.

We define the class of admissible controls $\mathfrak{A}_{ad}$ as follows
\begin{equation}
\label{1.2}
\mathfrak{A}_{ad}=\left\{A\in L^\infty(\Omega;\mathbb{S}^N) \left|
\begin{array}{c}
\xi^2_1|\eta|^2 \le (\eta, A\eta)_{\mathbb{R}^N} \le \xi^2_2 |\eta|^2\,
\text{a.e.}\,\mbox {in}\,\Omega,\,  \forall\, \eta\in \mathbb{R}^N,\\[1ex]
A^\frac{1}{2}\in BV(\Omega;\mathbb{S}^N ),\; \int_\Omega |D A^\frac{1}{2}|\le\gamma,
\end{array}
\right.\right\}
\end{equation}
where $\gamma>0$ is a given constant.
In view of estimate
\begin{equation*}
\|A^\frac{1}{2}(x)\|_{\mathbb{S}^N}\le \sqrt{N}\, \|A^\frac{1}{2}(x)\|_2\le \sqrt{N}\, \xi_2(x)\ \text{ a.e. in }\Omega,
\end{equation*}
it is clear that $\mathfrak{A}_{ad}$ is a nonempty convex subset of $L^\infty(\Omega;\mathbb{S}^N)$.

\textit{Anisotropic Laplace operator.}
Let us consider now  the nonlinear operator $\mathcal{A}(A,y):L^\infty(\Omega;\mathbb{S}^N)\times W_0^{1,p}(\Omega)\rightarrow W^{-1,q}(\Omega)$ defined as
$$\mathcal{A}(A,y)=-\mathrm{div}\big(|(A\nabla y,\nabla y)_{\mathbb{R}^N} |^{\frac{p-2}{2}}A\nabla y\big)$$
or via the pairing
\begin{multline}\label{1.3}
\langle \mathcal{A}(A,y),v\rangle_{W^{-1,q}(\Omega);W^{1,p}_0(\Omega)}:= \int_\Omega |(A\nabla y,\nabla y)_{\mathbb{R}^N} |^{\frac{p-2}{2}}\left(A\nabla y,\nabla v\right)_{\mathbb{R}^N} \,dx\\
=\int_\Omega |A^{\frac{1}{2}}\nabla y|^{p-2}\left(A\nabla y,\nabla v\right)_{\mathbb{R}^N} \,dx,\quad\forall v\in W_0^{1,p}(\Omega).
\end{multline}
\begin{definition}
\label{Def 1.1} We say that a function $y=y(A,f)$ is  a weak
solution (in the sense of Minty) to boundary value problem
\begin{gather}\label{1.7}
-\mathrm{div}\big(|(A\nabla y,\nabla y)_{\mathbb{R}^N} |^{\frac{p-2}{2}}A\nabla y\big) =f\text{ in }\Omega,\\\label{1.8}
A\in\mathfrak{A}_{ad},\; y\in W_0^{1,p}(\Omega),
\end{gather}
for
a fixed control $A\in \mathfrak{A}_{ad}$ and given function $f\in
L^2(\Omega)$ if
the inequality
\begin{equation}
\label{1.6*}
\int_\Omega |A^{1/2}\nabla \varphi|_{\mathbb{R}^N}^{p-2}(A\nabla\varphi, \nabla\varphi-\nabla y)_{\mathbb{R}^N}\,dx\ge\int_\Omega f(\varphi-y)\,dx
\end{equation}
holds for any $\varphi\in C^\infty_0(\Omega)$.
\end{definition}
\begin{remark}
\label{Rem 1.1}
Another definition of the weak solution to the considered boundary value problem appears more natural:
\begin{equation*}
\int_\Omega |A^{1/2}\nabla y|_{\mathbb{R}^N}^{p-2}(A\nabla y,\nabla \varphi)_{\mathbb{R}^N}\,dx=\int_\Omega f\varphi\,dx\;\forall\,\varphi\in C_0^\infty(\Omega).
\end{equation*}
However, both concepts for the weak solutions coincide (see, for instance, \cite{Past}).
\end{remark}

Let us show that for each $A\in \mathfrak{A}_{ad}$ operator $\mathcal{A}(\cdot)=\mathcal{A}(A,\cdot):W^{1,p}_0(\Omega)\to W^{-1,q}(\Omega)$ is strictly monotone, coercive and semi-continuous, where the above mentioned properties  have respectively the following meaning:
\begin{gather}
\label{1.4}
\left<\mathcal{A}y-\mathcal{A}v,y-v\right>_{W^{-1,q}(\Omega);W^{1,p}_0(\Omega)}\ge 0,\quad\forall\,y,v\in W_0^{1,p}(\Omega);\\
\label{1.4a}
\left<\mathcal{A}y-\mathcal{A}v,y-v\right>_{W^{-1,q}(\Omega);W^{1,p}_0(\Omega)}=0\ \Longrightarrow \ y=v;\\
\label{1.5}
\left<\mathcal{A}y,y\right>_{W^{-1,q}(\Omega);W^{1,p}_0(\Omega)}\to +\infty\text{ provided } \|y\|_{W_0^{1,p}(\Omega)\to\infty};\\
\label{1.6}
\mathbb{R}\ni t\mapsto \left<\mathcal{A}(y+tv),w\right>_{W^{-1,q}(\Omega);W^{1,p}_0(\Omega)}\text{ is continuous }\forall\,y,v,w\in W_0^{1,p}(\Omega).
\end{gather}
Indeed, the right-hand side of \eqref{1.3} is continuous with respect to $v\in W^{1,p}_0(\Omega)$ and, therefore, represents an element of $W^{-1,q}(\Omega)$ because
\begin{gather*}
\int_\Omega |A^{\frac{1}{2}}\nabla y|^{p-2}\left(A\nabla y,\nabla v\right)_{\mathbb{R}^N}dx \le
\left(\int_\Omega |A^{\frac{1}{2}}\nabla y|^{p}dx \right)^{\frac{p-1}{p}}
\left(\int_\Omega |A^{\frac{1}{2}}\nabla v|^{p}dx \right)^{\frac{1}{p}}\\
\le \|\xi_2\|^{p}_{L^\infty(\Omega)} \|\nabla y\|^{p-1}_{L^p(\Omega)^N}\|\nabla v\|_{L^p(\Omega)^N}=
\|\xi_2\|^{p}_{L^\infty(\Omega)} \|y\|^{p-1}_{W^{1,p}_0(\Omega)}\| v\|_{W^{1,p}_0(\Omega)}
\end{gather*}
(we apply here the H\"{o}lder's inequality and  estimate
$|A^{\frac{1}{2}}\nabla \varphi|^p\le \xi_2^{p} |\nabla \varphi|^p$
coming from the condition $A\in \mathfrak{A}_{ad}$). Hence, for each $A\in \mathfrak{A}_{ad}$  the operator
$\mathcal{A}(A,\cdot):W^{1,p}_0(\Omega)\to W^{-1,q}(\Omega)$ is bounded.
The coercivity property of $\mathcal{A}$ we get immediately, since
\begin{equation*}
\left<\mathcal{A}y,y\right>_{W^{-1,q}(\Omega);W^{1,p}_0(\Omega)}\ge \alpha^{p}\|y\|^p_{W_0^{1,p}(\Omega)}.
\end{equation*}
As for the proof of the strict monotonicity and semicontinuity of the operator $\mathcal{A}$, we refer
for the details to \cite{Lions69,Roubicek}).
Then, by well known existence results for non-linear elliptic equations with coercive, semi-continuous, strictly monotone operators, the Dirichlet boundary value problem \eqref{1.7}--\eqref{1.8}
admits a unique weak solution for every fixed control matrix $A\in \mathfrak{A}_{ad}$ and every distribution $f\in L^2(\Omega)$.

\textit{On equations of Hammerstein type.}
Let $Y$ and $Z$ be Banach spaces, let $Y_0\subset Y$ be an arbitrary bounded set, and let $Z^\ast$ be the dual space to $Z$. To begin with we recall some useful properties of non-linear operators, concerning the solvability problem for  Hammerstein type equations and systems.
\begin{definition}\label{Def.1}
We say that the operator $G:D(G)\subset Z\to Z^\ast$ is radially continuous if for any $z_1,z_2\in X$ there exists $\e>0$ such that $z_1+\tau z_2\in D(G)$ for all $\tau\in [0,\e]$ and a real-valued function $[0,\e]\ni \tau\to\langle G(z_1+\tau z_2),z_2\rangle_{Z^\ast;Z}$ is continuous.
\end{definition}
\begin{definition}
\label{Def.2}
An operator $G:Y\times Z\to Z^\ast$ is said to have a uniformly semi-bounded variation (u.s.b.v.) if for any bounded set  $Y_0\subset Y$
and any elements $z_1,z_2\in D(G)$ such that $\|z_i\|_Z\leq R$, $i=1,2$, the following inequality
\begin{equation}
\label{1.9.1}
\langle G(y,z_1)-G(y,z_2), z_1-z_2\rangle_{Z^\ast;Z}\ge -\inf_{y\in Y_0}{C_{y}(R;\||z_1-z_2\||_Z)}
\end{equation}
holds true provided  the function $C_{y}:\mathbb{R}_+\times\mathbb{R}_+\to \mathbb{R}$ is continuous for each element $y\in Y_0$, and
$\ds\frac{1}{t}C_{y}(r,t)\to 0$ as $t\to 0$, $\forall\, r>0$. Here, $\||\cdot\||_Z$ is a seminorm on $Z$ such that $\||\cdot\||_Z$ is compact with respect to the norm $\|\cdot\|_Z$.
\end{definition}

It is worth to note that Definition \ref{Def.2} gives in fact a certain generalization of the classical monotonicity property. Indeed, if $C_{y}(\rho,r)\equiv 0$, then \eqref{1.9.1} implies the monotonicity property for the operator $G$ with respect to the second argument.
\begin{remark}\label{Rem 1.5}
Each operator $G:Y\times Z\to Z^\ast$ with u.s.b.v. possesses the following property (see for comparison Remark 1.1.2 in  \cite{AMJA}): if a set $K\subset Z$ is such that $\|z\|_{Z}\le k_1$ and $\langle G(y,z),z\rangle_{Z^\ast;Z}\le k_2$ for all $z\in K$ and $y\in Y_0$, then there exists a constant $C>0$ such that $\|G(y,z)\|_{Z^\ast}\le C$, $\forall\,z\in K$ and $\forall y\in Y_0$.
\end{remark}

Let $B:Z^\ast\to Z$ and $F:Y\times Z\to Z^\ast$ be given operators such that the mapping $Z^\ast\ni z^\ast\mapsto B(z^\ast)\in Z$ is linear. Let $g\in Z$ be a given distribution. Then a typical operator equation of Hammerstein type can be represented as follows
\begin{equation}
\label{1.9.2}
z+B F(y,z)=g.
\end{equation}
The following existence result is well-known (see \cite[Theorem 1.2.1]{AMJA}).
\begin{theorem}\label{Th 1.1*}
Let $B:Z^\ast\to Z$ be a linear continuous positive operator such that there exists a right inverse operator $B^{-1}_{r}:Z\to Z^\ast$. Let $F:Y\times Z\to Z^\ast$ be an operator with u.s.b.v such that
$F(y,\cdot):Z\to Z^\ast$ is radially continuous for each $y\in Y_0$ and the following inequality holds true
$$
\langle F(y,z)-B^{-1}_{r}g,z\rangle_{Z^\ast;Z}\ge 0\quad\mbox{ if only } \|z\|_{Z}>\lambda>0,\;\lambda = const.
$$
Then the set
$$
\mathcal{H}(y)=\{z\in Z:\;z+BF(y,z)=g\ \text{in the sense of distributions }\}
$$
is non-empty and weakly compact for every fixed $y\in Y_0$ and $g\in Z$.
\end{theorem}

In what follows, we set $Y=W_0^{1,p}(\Omega)$, $Z=L^p(\Omega)$, and $Z^\ast=L^q(\Omega)$.


\section{Setting of the optimal control problem }
\label{Sec 2}
Let us consider the following optimal control problem:
\begin{gather}
\label{2.7} \text{Minimize }\ \Big\{I(A,y,z)=\int_\Omega
|z(x)-z_d(x)|^2\,dx\Big\},
\end{gather}
subject to the constraints
\begin{gather}
\label{2.7a} \int_\Omega |A^{1/2}\nabla y|^{p-2}(A\nabla y,\nabla \varphi)_{\mathbb{R}^N}\,dx=\int_\Omega f \varphi\,dx,\  \forall\,\varphi\in W^{1,p}_0(\Omega),\\
\label{2.7b} A\in\mathfrak{A}_{ad},\quad y\in W^{1,p}_0(\Omega),\\
\label{2.7c} \int_\Omega z \,\phi \, dx + \int_\Omega B F(y,z)\,\phi\,dx=0,\  \forall\,\phi\in L^q(\Omega),
\end{gather}
where $f\in L^2(\Omega)$ and $z_d\in L^p(\Omega)$ are given distributions, $B:L^q(\Omega)\to L^p(\Omega)$ is a linear operator, $F:W_0^{1,p}(\Omega)\times L^p(\Omega)\to L^q(\Omega)$ is a non-linear operator.

Let us denote by $\Xi\subset L^\infty(\Omega;\mathbb{S}^{N})\times W^{1,p}_0(\Omega)\times L^p(\Omega)$ the set of all admissible triplets to the
optimal control problem \eqref{2.7}--\eqref{2.7c}.

Hereinafter we suppose that the space $ L^\infty(\Omega;\mathbb{S}^N) \times W_0^{1,p}(\Omega)\times L^p(\Omega)$ is endowed with the norm $\|(A,y,z)\|_{ L^\infty(\Omega;\mathbb{S}^N) \times W_0^{1,p}(\Omega)\times L^p(\Omega)}:=\|A^\frac{1}{2}\|_{BV(\Omega;\mathbb{S}^N)}+\|y\|_{W^{1,p}_0(\Omega)}+\|z\|_{L^p(\Omega)}$.
\begin{remark}
\label{Rem 3.9}
We recall that a sequence $\{f_k\}_{k=1}^\infty$ converges weakly-$^\ast$ to $f$ in $BV(\Omega)$ if and only if the two following conditions hold (see \cite{AFP2000}): $f_k\to f$ strongly in $L^1(\Omega)$ and $D f_k\stackrel{\ast}{\rightharpoonup} Df$ weakly$^\ast$ in the space of Radon measures $\mathcal{M}(\Omega;\mathbb{R}^N)$.  Moreover, if $\{f_k\}_{k=1}^\infty\subset BV(\Omega)$ converges strongly to some $f$ in $L^1(\Omega)$ and satisfies $\sup_{k\in \mathbb{N}}\int_\Omega|Df_k|<+\infty$, then (see, for instance,  \cite{AFP2000})
\begin{equation}
\label{3.9}
\begin{aligned}
(i)&\  f\in BV(\Omega)\ \text{ and }\ \int_\Omega|Df|\le\liminf_{k\to\infty}\int_\Omega|Df_k|;\\
(ii)& \ f_k\stackrel{\ast}{\rightharpoonup} f\ \text{ in }\ BV(\Omega).
\end{aligned}
\end{equation}
Also we recall, that uniformly bounded sets in $BV$-norm are relatively
compact in $L^1(\Omega)$.
\end{remark}

\begin{definition}\label{def_tau}
 We say that a sequence of triplets
$\{(A_k,y_k,z_k)\}_{k\in\mathbb{N}}$ from the space $ L^\infty(\Omega;\mathbb{S}^{N})\times W^{1,p}_0(\Omega)\times L^p(\Omega)$
$\tau$-converges to a triplet $(A_0,y_0,z_0)$ if $A_k^{\frac{1}{2}}\stackrel{\ast}{\rightharpoonup} A_0^{\frac{1}{2}}$ in $BV(\Omega;\mathbb{S}^N)$,
$y_k\rightharpoonup y_0$ in $W_0^{1,p}(\Omega)$ and $z_k\rightharpoonup z_0$ in $L^p(\Omega)$.
\end{definition}

Further we use the following auxiliary results.
\begin{proposition}
\label{Prop 1.16} For each $A\in\mathfrak{A}_{ad}$ and every $f\in
L^2(\Omega)$, a weak solution $y$ to variational
problem \eqref{2.7a}--\eqref{2.7b} satisfies the estimate
\begin{equation}
\label{1.17}
\|y\|_{W^{1,p}_0(\Omega)}\le
 \alpha^{-{q}}\|f\|^{\frac{q}{p}}_{W^{-1,\,q}(\Omega)}.
\end{equation}
\end{proposition}
\begin{proof} The estimate \eqref{1.17} immediately follows from the following relations
\begin{align*}
\alpha^{p}\|y\|^p_{W_0^{1,p}(\Omega)}&\le\int_\Omega |A^{\frac{1}{2}}\nabla y|^p \,dx =\left<\mathcal{A}(A,y),y\right>_{W^{-1,q}(\Omega);W_0^{1,p}(\Omega)}
\\&=\left<f,y\right>_{W^{-1,q}(\Omega);W_0^{1,p}(\Omega)}
\le\|f\|_{W^{-1,q}(\Omega)}\|y\|_{W_0^{1,p}(\Omega)}.
\end{align*}
\end{proof}

\begin{lemma}
\label{Lemma 3.0a}
Let $\left\{(A_k,y_k)\in  L^\infty(\Omega;\mathbb{S}^N)\times W_0^{1,p}(\Omega) \right\}_{k\in \mathbb{N}}$ be a sequence of pairs such that $A_k\in\mathfrak{A}_{ad}$ $\forall\,k\in\mathbb{N}$, $A^\frac{1}{2}_k\stackrel{\ast}{\rightharpoonup} A^\frac{1}{2}$
in $BV(\Omega;\mathbb{S}^N)$, and
$y_k\rightharpoonup y$ in $W^{1,p}_0(\Omega)$. Then
\begin{align}
\notag
\lim_{k\to\infty}&\int_\Omega |\left(\nabla\varphi,A_k\nabla\varphi\right)_{\mathbb{R}^N} |^\frac{p-2}{2}\left(\nabla y_k,A_k\nabla\varphi\right)_{\mathbb{R}^N} \,dx\\
\label{3.00a}
=&\int_\Omega |\left(\nabla\varphi,A\nabla\varphi\right)_{\mathbb{R}^N} |^\frac{p-2}{2}\left(\nabla y,A\nabla\varphi\right) _{\mathbb{R}^N} \,dx,\quad\forall\varphi\in C^{\infty}_0(\Omega).
\end{align}
\end{lemma}
\begin{proof}
Since $A^\frac{1}{2}_k\rightarrow A^\frac{1}{2}$ in $L^1(\Omega;\mathbb{S}^N)$ and $\{A_k\}_{k \in \mathbb{N}}$ is bounded in $L^\infty(\Omega;\mathbb{S}^N)$, by Lebesgue's Theorem we get that $A^\frac{1}{2}_k\rightarrow A^\frac{1}{2}$ strongly in $L^r(\Omega;\mathbb{S}^N)$ for every $1 \le r < +\infty$. Hence, $A^\frac{1}{2}_k\nabla\varphi\rightarrow A^\frac{1}{2}\nabla\varphi$ strongly in $L^p(\Omega)^N$ for every $\varphi\in C^{\infty}_0(\Omega)$. Therefore,
\begin{equation}
\label{3.01a}
|A^{\frac{1}{2}}_k\nabla\varphi|^{p-2}A^\frac{1}{2}_k\nabla\varphi\rightarrow |A^{\frac{1}{2}}\nabla\varphi|^{p-2}A^\frac{1}{2}\nabla\varphi\quad \text{ in }\ L^q(\Omega)^N,\quad\forall\,\varphi\in C^{\infty}_0(\Omega).
\end{equation}

Moreover, since $A^\frac{1}{2}_k\nabla\psi\rightarrow A^\frac{1}{2}\nabla\psi$ strongly in $L^q(\Omega)^N$ for every $\psi\in C^{\infty}_0(\Omega)$ and $\nabla y_k\rightharpoonup \nabla y$ in $L^p(\Omega)^N$, it follows that
\begin{align}
\notag
\int_\Omega \left(A^\frac{1}{2}_k\nabla y_k,\nabla\psi\right)_{\mathbb{R}^N} \,dx &= \int_\Omega \left(\nabla y_k,A^\frac{1}{2}_k\nabla\psi\right)_{\mathbb{R}^N} \,dx
\\\label{3.02a}\rightarrow
\int_\Omega \left(\nabla y,A^\frac{1}{2}\nabla\psi\right)_{\mathbb{R}^N} \,dx
&= \int_\Omega \left(A^\frac{1}{2}\nabla y,\nabla\psi\right)_{\mathbb{R}^N} \,dx,\quad\forall\psi\in C^\infty_0(\Omega)
\end{align}
as a product of weakly and strongly convergent sequences in $L^p(\Omega)^N$ and $L^q(\Omega)^N$, respectively.
Using the fact that
\[
\sup_{k\in \mathbb{N}}\|A^\frac{1}{2}_k\nabla y_k\|_{L^p(\Omega)^N}\le \|\xi_2\|_{L^\infty(\Omega)}
\sup_{k\in \mathbb{N}}\|\nabla y_k\|_{L^p(\Omega)^N}<+\infty,
\]
we finally get from \eqref{3.02a}
\begin{equation}
\label{3.03a}
A^\frac{1}{2}_k\nabla y_k\rightharpoonup A^\frac{1}{2}\nabla y\quad\text{in }\ L^p(\Omega)^N.
\end{equation}
Thus, to complete the proof it remains to note that
\begin{multline*}
\int_\Omega |\left(\nabla\varphi,A_k\nabla\varphi\right)_{\mathbb{R}^N} |^\frac{p-2}{2}\left(\nabla y_k,A_k\nabla\varphi\right)_{\mathbb{R}^N}dx \\=
\int_\Omega \left(|A^{\frac{1}{2}}_k\nabla\varphi|^{p-2}A^{\frac{1}{2}}_k\nabla \varphi,A^\frac{1}{2}_k\nabla y_k\right)_{\mathbb{R}^N}dx
\end{multline*}
and apply the properties \eqref{3.01a} and \eqref{3.03a}.
\end{proof}

The following result concerns the regularity of the optimal control problem \eqref{2.7}--\eqref{2.7c}.
\begin{proposition}\label{prop 1.15} Let
 $B:L^q(\Omega)\to L^p(\Omega)$ and $F:W_0^{1,p}(\Omega)\times L^p(\Omega)\to L^q(\Omega)$ be operators satisfying all conditions of Theorem \ref{Th 1.1*}.
Then the set
\begin{multline*}
\Xi=\big\{(A,y,z)\in L^\infty(\Omega;\mathbb{S}^N)\times W_0^{1,p}(\Omega)\times L^p(\Omega): \\\mathcal{A}(A,y)=f,\; z+B F(y,z)=0\big\}
\end{multline*}
is nonempty for every $f\in L^2(\Omega)$.
\end{proposition}
\begin{proof} Let $A\in\mathfrak{A}_{ad}$ be an arbitrary admissible control. Then for a given $f\in L^2(\Omega)$, the Dirichlet boundary problem \eqref{2.7a}--\eqref{2.7b} admits a unique solution $y_A=y(A,f)\in W_0^{1,p}(\Omega)$ which satisfies the estimate \eqref{1.17}. It remains to remark that the corresponding Hammerstein equation $z+BF(y_A,z)=0$ has a nonempty set of solutions $\mathcal{H}(y_A)$ by Theorem \ref{Th 1.1*}.
\end{proof}

\section{Existence of optimal solutions}

The following result is crucial for our consideration and it states the fact, that the set of admissible triplets to the optimal control problem  \eqref{2.7}--\eqref{2.7c} is closed with respect to $\tau$-topology of the space $L^\infty(\Omega;\mathbb{S}^N)\times W_0^{1,p}(\Omega)\times L^p(\Omega)$.

\begin{theorem}
\label{Th 2.8} Assume the following conditions hold:
\begin{itemize}
\item The operators $B:L^q(\Omega)\to L^p(\Omega)$ and $F:W_0^{1,p}(\Omega)\times L^p(\Omega)\to L^q(\Omega)$ satisfy conditions of Theorem \ref{Th 1.1*};
\item The operator $F(\cdot, z):W_0^{1,p}(\Omega)\to L^q(\Omega)$ is compact in the following sense:

if $y_k\rightharpoonup y_0$ weakly in $W_0^{1,p}(\Omega)$, then  $F(y_k,z)\to F(y_0,z)$ strongly in $L^q(\Omega)$.
\end{itemize}
Then for every $f\in L^2(\Omega)$ the set $\Xi$
is sequentially $\tau$-closed, i.e. if a sequence $\{(A_k,y_k,z_k)\in \Xi\}_{k\in\mathbb{N}}$ $\tau$-converges to a triplet $(A_0,y_0,z_0)\in L^\infty(\Omega;\mathbb{S}^N)\times W^{1,p}_0(\Omega)\times L^p(\Omega)$, then $A_0\in \mathfrak{A}_{ad}$, $y_0=y(A_0)$, $z_0\in\mathcal{H}(y_0)$, and, therefore,
$(A_0,y_0,z_0)\in\Xi$.
\end{theorem}
\begin{proof}
Let $\{(A_k,y_k,z_k)\}_{k\in\mathbb{N}}\subset \Xi$ be any $\tau$-convergent sequence of admissible triplets to the optimal control problem \eqref{2.7}--\eqref{2.7c}, and let $(A_0,y_0,z_0)$ be its $\tau$-limit in the sense of Definition \ref{def_tau}. We divide the rest of the proof onto two steps.

\textit{Step 1.} On this step we show that $A_0\in \mathfrak{A}_{ad}$ and $y_0=y(A_0)$. As follows from Definition \ref{def_tau} and  Remark \ref{Rem 3.9}, we have
\begin{gather}
A^\frac{1}{2}_k\to A_0^\frac{1}{2}\text{ in }L^1(\Omega;\mathbb{S}^N),\quad
\label{3.12b}
y_k\rightharpoonup y_0\ \text{ in }\ W^{1,p}_0(\Omega),\\
\label{3.12bb}
\quad A^\frac{1}{2}_k\to A_0^\frac{1}{2}\text{ almost everywhere in }\ \Omega,\\
\label{3.12bbb}
\int_\Omega|D A_0^\frac{1}{2}|\le\liminf_{k\to\infty}\int_\Omega|D A^\frac{1}{2}_k|\le \gamma.
\end{gather}
 Moreover, as follows from \eqref{3.12bb} and definition of the set $\mathfrak{A}_{ad}$ (see \eqref{1.2}), the inequality
\begin{equation}
\label{3.12bc}
\xi^2_1|\eta|^2 \le ( \eta, A_0 \eta)_{\mathbb{R}^N} \le \xi^2_2 |\eta|^2
\text{ a.e. in }\Omega \  \forall\, \eta\in \mathbb{R}^N,
\end{equation}
is valid. Thus,  $A\in \mathfrak{A}_{ad}$.
Hence, it is enough to show that the limit pair $(A_0,y_0)$ is related by \eqref{2.7a} or \eqref{1.6*} (see Definition \ref{Def 1.1} and Remark \ref{Rem 1.1}).  With that in mind we write down relation \eqref{1.6*} for $(A_k,y_k)$ and arbitrary $\varphi\in C_0^\infty(\Omega)$:
\begin{equation}
\label{3.13}
\int_\Omega |(A_k\nabla \varphi,\nabla \varphi)_{\mathbb{R}^N} |^{\frac{p-2}{2}}\left(A_k\nabla \varphi,\nabla \varphi-\nabla y_k\right)_{\mathbb{R}^N}dx\ge\int_\Omega {f}(\varphi- y_k)dx,
\end{equation}
and pass to the limit in it as $k\to\infty$.

In view of the properties \eqref{3.12b}--\eqref{3.12bc} and the boundedness of $\{A_k\}_{k \in \mathbb{N}}$ in $L^\infty(\Omega;\mathbb{S}^N)$, by Lebesgue's Theorem we get that
 $A^\frac{1}{2}_k\nabla\varphi\rightarrow A_0^\frac{1}{2}\nabla\varphi$ strongly in $L^p(\Omega)^N$ for every $\varphi\in C^{\infty}_0(\Omega)$. Therefore,
\begin{align*}
\lim_{k\to\infty}&\int_\Omega |(A_k\nabla \varphi,\nabla \varphi)_{\mathbb{R}^N} |^{\frac{p-2}{2}}\left(A_k\nabla \varphi,\nabla \varphi\right)_{\mathbb{R}^N}dx=\lim_{k\to\infty} \int_\Omega |A_k^\frac{1}{2}\nabla \varphi|^{p}\,dx\\
&= \int_\Omega |A_0^\frac{1}{2}\nabla \varphi|^{p}\,dx
=
\int_\Omega |(A_0\nabla \varphi,\nabla \varphi)_{\mathbb{R}^N} |^{\frac{p-2}{2}}\left(A_0\nabla \varphi,\nabla \varphi\right)_{\mathbb{R}^N}dx
\end{align*}
and, by Lemma~\ref{Lemma 3.0a},
\begin{multline*}
\lim_{k\to\infty} \int_\Omega |(A_k\nabla \varphi,\nabla \varphi)_{\mathbb{R}^N} |^{\frac{p-2}{2}}\left( A_k\nabla \varphi, \nabla y_k\right)_{\mathbb{R}^N} \,dx\\= \int_\Omega |(A_0\nabla \varphi,\nabla \varphi)_{\mathbb{R}^N} |^{\frac{p-2}{2}}\left( A_0\nabla \varphi,\nabla y_0\right)_{\mathbb{R}^N} \,dx.
\end{multline*}
We, thus, can pass to the limit in relation \eqref{3.13} as $k\to\infty$ and arrive at the inequality
\begin{multline*}
\int_\Omega |(A_0\nabla \varphi,\nabla \varphi)_{\mathbb{R}^N} |^{\frac{p-2}{2}}\left(A_0\nabla \varphi,\nabla \varphi-\nabla y_0\right)_{\mathbb{R}^N}dx \\ \label{3.14}\ge\int_\Omega {f}(\varphi- y_0)\,dx,\; \forall\,\varphi\in C_0^\infty(\Omega),
\end{multline*}
which means that $y_0\in W_0^{1,p}(\Omega)$ is a solution to boundary value problem \eqref{2.7a}--\eqref{2.7b}, corresponding to control matrix $A_0$. This fact together with $A_0\in \mathfrak{A}_{ad}$ leads us to the conclusion: $y_0=y(A_0)$.

\textit{Step 2.} On this step we show that  $z_0\in\mathcal{H}(y_0)$. To this end, we have to pass to the limit in equation
\begin{equation}\label{1.22}
z_k+BF(y_k,z_k)=0
\end{equation}
as $k\to\infty$ and get the limit pair $(y_0,z_0)$ is related by the equation
$
z_0+BF(y_0,z_0)=0.
$
With that in mind, let us rewrite equation  \eqref{1.22} in the following way
$$
B^\ast w_k +BF(y_k,B^\ast w_k)=0,
$$
where $w_k\in L^q(\Omega)$, $B^\ast:L^q(\Omega)\to L^p(\Omega)$ is the conjugate operator for $B$, i.e. $\langle B\nu, w\rangle_{L^p(\Omega);L^q(\Omega)}=\langle B^\ast w,\nu\rangle_{L^p(\Omega);L^q(\Omega)}$ and $B^\ast w_k=z_k$. Then, for every $k\in \mathbb{N}$, we have the equality
\begin{equation}
\label{1.18}
 \langle B^\ast w_k,w_k\rangle_{L^p(\Omega);L^q(\Omega)}=-\langle F(y_k,B^\ast w_k),B^\ast w_k\rangle_{L^p(\Omega);L^q(\Omega)}.
 \end{equation}
 The left-hand side in \eqref{1.18} is strictly positive for every $w_k\neq 0$, hence,  the right-hand side  must be positive as well.
In view of the initial assumptions, namely,
$$
\langle F(y,u),u\rangle_{L^q(\Omega);L^p(\Omega)}\ge 0\ \text{ if only}\ \|u\|_{L^p(\Omega)}>\lambda,
$$
we conclude that
\begin{equation}
\label{7*}
\|B^\ast w_k\|_{L^p(\Omega)}=\|z_k\|_{L^p(\Omega)}\le \lambda.
\end{equation}
Since the linear positive operator $B^\ast$ cannot map unbounded sets into bounded ones, it follows that $\|w_k\|_{L^q(\Omega)}\le \lambda_1$.
As a result, see \eqref{1.18},  we have
$$\langle  F(y_k,B^\ast w_k),B^\ast w_k\rangle_{L^q(\Omega);L^p(\Omega)}\le c_1.$$  Hence, in view of Remark \ref{Rem 1.5}, we get
$$
\|F(y_k,B^\ast w_k)\|_{L^q(\Omega)}=\|F(y_k,z_k)\|_{L^q(\Omega)}\le c_2\
\text{ as }\ \|z_k\|_{L^p(\Omega)}\le \lambda.
$$
Since the left-hand side of \eqref{1.18}  does not depend on $y_k$, it follows that the constant $c_2>0$ does not depend on $y_k$ as well.

Taking these arguments into account, we may suppose existence of an element $\nu_0\in L^q(\Omega)$ such that up to a subsequence the weak convergence $F(y_k,z_k)\rightharpoonup \nu_0$ in $L^q(\Omega)$ takes place. As a result, passing to the limit in  \eqref{1.22}, by continuity of $B$, we finally get
\begin{equation}\label{1.20***}
z_0+B\nu_0=0.
\end{equation}
It remains to show that $\nu_0=F(y_0,z_0)$. Let us take an arbitrary element $z\in L^p(\Omega)$ such that $\|z\|_{L^p(\Omega)}\le \lambda$. Using the fact that $F$ is an operator with u.s.b.v., we have
$$
\langle F(y_k,z)-F(y_k,z_k),z-z_k\rangle_{L^q(\Omega);L^p(\Omega)}\ge  -\inf_{y_k\in Y_0}C_{y_k}(\lambda;\||z-z_k\||_{L^p(\Omega)}),
$$
where $Y_0=\{y\in W_0^{1,p}(\Omega):\;y\mbox{ satisfies }\eqref{1.17}\}$, or, after transformation,
\begin{multline}
\label{1.21***}
\langle F(y_k,z),z-z_k\rangle_{L^q(\Omega);L^p(\Omega)}-\langle F(y_k,z_k),z\rangle_{L^q(\Omega);L^p(\Omega)}\\\ge \langle F(y_k,z_k),-z_k\rangle_{L^q(\Omega);L^p(\Omega)}-\inf_{y_k\in Y_0}C_{y_k}(\lambda;\||z-z_k\||_{L^p(\Omega)}).
\end{multline}
Since $-z_k=BF(y_k,z_k)$, it follows from \eqref{1.21***} that
\begin{multline}
\label{1.21.0}
\langle F(y_k,z),z-z_k\rangle_{L^q(\Omega);L^p(\Omega)}-\langle F(y_k,z_k),z\rangle_{L^q(\Omega);L^p(\Omega)}  \\\ge \langle F(y_k,z_k),B F(y_k,z_k) \rangle_{L^q(\Omega);L^p(\Omega)}-\inf_{y_k\in Y_0}C_{y_k}(\lambda;\||z-z_k\||_{L^p(\Omega)}).
\end{multline}

In the meantime, due to the weak convergence $F(y_k,z_k)\rightharpoonup \nu_0$ in $L^q(\Omega)$ as $k\to\infty$, we arrive at the following obvious properties
\begin{gather}
\label{1.21.1}
\liminf_{k\to\infty}\langle F(y_k,z_k), B F(y_k,z_k) \rangle_{L^q(\Omega);L^p(\Omega)}\ge \langle \nu_0, B\nu_0\rangle_{L^q(\Omega);L^p(\Omega)},\\
\label{1.21.2a}
\lim_{k\to\infty}\langle F(y_k,z_k) ,z \rangle_{L^q(\Omega);L^p(\Omega)}=
\langle \nu_0, z\rangle_{L^q(\Omega);L^p(\Omega)}.
\end{gather}
Moreover, the continuity of the function $C_{y_k}$ with respect to the second argument and the compactness property of operator $F$, which means strong convergence $F(y_k,z)\to F(y_0,z)$ in $L^q(\Omega)$, lead to the conclusion
\begin{gather}
\label{1.21.3}
\lim_{k\to\infty} C_{y}(\lambda;\||z-z_k\||_{L^p(\Omega)})= C_{y}(\lambda;\||z-z_0\||_{L^p(\Omega)}),\quad\forall\,y\in Y_0,\\
\label{1.21.4}
\lim_{k\to\infty}\langle F(y_k,z),z-z_k\rangle_{L^q(\Omega);L^p(\Omega)}= \langle F(y_0,z),z-z_0\rangle_{L^q(\Omega);L^p(\Omega)}.
\end{gather}

As a result, using the properties \eqref{1.21.1}--\eqref{1.21.4}, we can pass to the limit in \eqref{1.21.0} as $k\to\infty$.  One gets
\begin{multline}
\langle F(y_0,z),z-z_0\rangle_{L^q(\Omega);L^p(\Omega)}-\langle \nu_0,z+B\nu_0\rangle_{L^q(\Omega);L^p(\Omega)}\\\label{1.21.5}\ge -\inf_{y\in Y_0}C_{y}(\lambda;\||z-z_0\||_{L^p(\Omega)}).
\end{multline}
Since $B\nu_0=-z_0$ by \eqref{1.20***}, we can rewrite the inequality \eqref{1.21.5} as follows
\begin{equation*}
\langle F(y_0,z)-\nu_0,z-z_0\rangle_{L^q(\Omega);L^p(\Omega)}\ge -\inf_{y\in Y_0}C_{y}(\lambda;\||z-z_0\||_{L^p(\Omega)}).
\end{equation*}
It remains to note that the operator $F$ is radially continuous for each $y\in Y_0$, and $F$ is the operator with u.s.b.v. (see Definitions \ref{Def.1} and \ref{Def.2}). Therefore, the last relation implies that $F(y_0,z_0)=\nu_0$ (see \cite[Theorem 1.1.2]{AMJA}) and, hence, equality \eqref{1.20***} finally takes the form
\begin{equation}\label{3*}
z_0+BF(y_0,z_0)=0.
\end{equation}
Thus, $z_0\in\mathcal{H}(y_0)$ and the triplet $(\mathcal{U}_0,y_0,z_0)$ is admissible for OCP \eqref{2.7}--\eqref{2.7c}. The proof is complete.
\end{proof}

\begin{remark}\label{Rem 1.7} In fact, as
follows from the proof of Theorem \ref{Th 2.8}, the set of admissible solutions $\Xi$ to the problem \eqref{2.7}--\eqref{2.7c} is sequentially $\tau$-compact.
To prove this fact it is enough to show the sequential compactness of the set of admissible controls with respect to the mentioned topology. Indeed, the set $\mathfrak{A}_{ad}$ is bounded in $L^\infty(\Omega; \mathbb{S}^N)$, so any sequence $\{A_k\}_{k\in\mathbb{N}}\subset \mathfrak{A}_{ad}$ is weakly-$\ast$ relatively compact in $L^\infty(\Omega; \mathbb{S}^N)$. This implies (see \eqref{1.2})  boundedness of $\{A^{\frac{1}{2}}_k\}_{k\in\mathbb{N}}$ in $BV(\Omega;\mathbb{S}^N)$ within a subsequence and, according to Remark \ref{Rem 3.9}, there exist an element $A$ and a subsequence, still denoted by the same index, such that $A^{\frac{1}{2}}_k\stackrel{\ast}{\rightharpoonup} A^{\frac{1}{2}}$ in $BV(\Omega;\mathbb{S}^N)$. It is easy to see, that correspondent solutions of \eqref{2.7}--\eqref{2.7a} $y_k=y(A_k)$, due to estimate \eqref{1.17}, form a weakly compact sequence in $W_0^{1,p}(\Omega)$  and sequence ${z_k=z(y_k)}_{k\in\mathbb{N}}$ is bounded in $L^p(\Omega)$ (see the proof of Theorem \ref{Th 2.8}), hence, it is weakly compact as well.
\end{remark}

Now we are in a position to prove the existence result for the original optimal control problem \eqref{2.7}--\eqref{2.7c}.
\begin{theorem}
\label{Th 2.9} Assume that
$\mathfrak{A}_{ad}\ne\emptyset$ and operators $B:L^q(\Omega)\to L^p(\Omega)$ and $F:W_0^{1,p}(\Omega)\times L^p(\Omega)\to L^q(\Omega)$ satisfy preconditions of Theorem~\ref{Th 2.8}. Then the optimal
control problem \eqref{2.7}--\eqref{2.7c} admits at least one
solution
\begin{gather*}
(A^{opt}, y^{opt},z^{opt})\in \Xi\subset L^\infty(\Omega;\mathbb{S}^N)\times W^{1,p}_0(\Omega)\times L^p(\Omega),\\
I(A^{opt}, y^{opt},z^{opt})=\inf_{(A, y,z)\in \Xi}I(A,y,z)
\end{gather*}
for each $f\in L^2(\Omega)$ and $z_d\in L^p(\Omega)$.
\end{theorem}
\begin{proof}
Since the cost functional in \eqref{2.7} is bounded from below and, by Theorem~\ref{Th 1.1*}, the set of admissible solutions $\Xi$ is nonempty, there exists a sequence $\{(A_k,y_k,z_k)\}_{k\in \mathbb{N}}\subset \Xi$ such that
$$
\lim_{k\to\infty}I(A_k,y_k,z_k)=\inf_{(A,y,z)\in \Xi}I(A,y,z).
$$
As was mentioned in Remark \ref{Rem 1.7}, the set of admissible solutions $\Xi$ to the problem \eqref{2.7}--\eqref{2.7c} is sequentially $\tau$-compact. Hence, there exists an admissible solution $(A_0,y_0,z_0)$ such that, up to a subsequence, $(A_k,y_k,z_k)\,\stackrel{\tau}{\rightarrow}\, (A_0,y_0,z_0)$ as $k\to\infty$.
In order to show that $(A_0,y_0,z_0)$  is an optimal solution of problem \eqref{2.7}--\eqref{2.7c}, it remains to make  use of the lower semicontinuity of the cost functional with respect to the $\tau$-convergence
\begin{align*}
I(A_0,y_0,z_0)&\le\liminf_{m\to\infty}I(A_{k_m},y_{k_m},z_{k_m})\\
&=\lim_{k\to\infty}I(A_k,y_k,z_k)=\inf_{(A, y,z)\in \Xi}I(A,y,z).
\end{align*}
The proof is complete.
\end{proof}

\section{Regularization of OCP \eqref{2.7}--\eqref{2.7c}}
\label{Sec 3a}
 In this section we introduce the two-parameter regularization for a specific example of the considered optimization problem for the case when the terms $[A^{\frac{1}{2}}\nabla y]^2$, $|y|^2$ and $|z|^2$ may grow large. Indeed, this circumstance causes certain difficulties in the process of deriving optimality conditions. As a result, we show that in suitable topologies optimal solutions of regularized problems tend  to some optimal solutions of the initial problem.

 The Hammerstein equation \eqref{2.7c} in the initial optimal control problem \eqref{2.7}--\eqref{2.7c} is given in rather general framework, so, for the sake of convenience, in this section we choose operators $B$ and $F$ more specifically, however, preconditions of theorem Theorem~\ref{Th 2.9} are still satisfied.

 Let us take a linear bounded and positive operator $B:L^q(\Omega)\to L^p(\Omega)$ as follows
\begin{equation}\label{operator_B}
(Bu)(x)=\int_\Omega K(x,t)u(t)\,dt,
\end{equation}
where the kernel $K(x,t)$ is such that
\begin{equation}\label{K}
\int_\Omega\int_\Omega (K(x,t))^p\,dx\,dx\le C_1.
\end{equation}
\begin{remark}\label{B_prop}
In view of condition \eqref{K}, there exists a constant $C_2>0$ such that
 \[
 \int_\Omega\int_\Omega (K(x,t))^2\,dx\,dx\le C_2.
 \]
 Hence, the linear positive operator $B$, considered as a mapping from $L^2(\Omega)$ to $L^2(\Omega)$, still maintains positivity and boundedness properties.
\end{remark}

As for the nonlinear operator $F:W_0^{1,p}(\Omega)\times L^p(\Omega)\to L^q(\Omega)$, we specify it to the form $F(y,z)=|y|^{p-2}y+|z|^{p-2}z$. It is clear that in this case $F$ is strictly monotone, radially continuous with respect to second argument and compact with respect to the first argument. So, further we deal with  the following Hammerstein equation
$$
z(x)+\int_\Omega K(x,t)\left(|y(t)|^{p-2}y(t)+|z(t)|^{p-2}z(t)\right)\,dt=0\mbox{ in } \Omega.
$$
\begin{remark}\label{unique}
The above Hammerstein equation has a unique solution for each fixed $y\in W_0^1(\Omega)$. Indeed, if
$z_1, z_2\in L^p(\Omega)$ are two different solutions, corresponding  to $y$, then
$z_1-z_2=-B(|z_1|^{p-2}z_1-|z_2|^{p-2}z_2)$. Let us multiply this equality on $w_1-w_2\in L^q(\Omega)$, where $B^\ast w_1=z_1$ and $B^\ast w_2=z_2$. Positivity property of $B^\ast$ and strict monotonicity of $F(y,z)$ with respect to the second argument imply
\begin{align*}
0&\le\langle B^\ast(w_1-w_2),w_1-w_2\rangle_{L^p(\Omega);L^q(\Omega)}=\langle z_1-z_2,w_1-w_2\rangle_{L^p(\Omega);L^q(\Omega)}\\&=-\langle B(|z_1|^{p-2}z_1-|z_2|^{p-2}z_2),w_1-w_2\rangle_{L^p(\Omega);L^q(\Omega)}
\\&=-\langle |z_1|^{p-2}z_1-|z_2|^{p-2}z_2,B^\ast( w_1-w_2)\rangle_{L^q(\Omega);L^p(\Omega)}\\&=-\langle |z_1|^{p-2}z_1-|z_2|^{p-2}z_2,z_1-z_2\rangle_{L^q(\Omega);L^p(\Omega)}\le 0\quad \Longleftrightarrow\quad z_1=z_2.
\end{align*}
\end{remark}

Hence, the initial control problem takes the form
\begin{gather}\label{3.1*}
I(A,y,z)=\int_\Omega
|z(x)-z_d(x)|^2\,dx\rightarrow \min,\\
\label{3.1a*} \int_\Omega |A^{1/2}\nabla y|^{p-2}(A\nabla y,\nabla \varphi)_{\mathbb{R}^N}\,dx=\int_\Omega f \varphi\,dx,\  \forall\,\varphi\in W^{1,p}_0(\Omega),\\
\label{3.1b*} A\in\mathfrak{A}_{ad},\quad y\in W^{1,p}_0(\Omega),\\
\notag  \int_\Omega \left(\int_\Omega K(t,x)\left(|y(t)|^{p-2}y(t)+|z(t)|^{p-2}z(t)\right)dt\right)\phi \,dx\\\label{3.1c*}
+\int_\Omega z \phi \, dx =0,\;\forall\,\phi\in C_0^\infty(\Omega).
\end{gather}

 As was pointed out in \cite{Roubicek}, the anisotropic $p$-Laplacian $\Delta_{p}(A,y)$ provides an example of a quasi-linear operator in divergence form with a so-called degenerate nonlinearity for $p>2$. In this context we have non-differentiability of the state $y$ with respect to the matrix-valued control $A$. As follows from Theorem~\ref{Th 2.9}, this fact is not an obstacle to prove existence of considered optimal controls in the coefficients, but it causes certain difficulties when one is deriving the optimality conditions for this problem. To overcome this difficulty, in this section we introduce the  family of correspondent approximating control problems (see, for comparison, the approach of Casas and Fernandez \cite{CasasFernandez1991} for quasi-linear elliptic variational inequalities with a distributed control in the right hand side).

\begin{equation}\label{3a.1}
\text{Minimize }\quad I_{\e,k}(A,y,z)=
\int_\Omega |z(x)-z_d(x)|^2\, dx
\end{equation}
subject to the constraints
\begin{gather}
\label{3a.2}
\mathcal{A}_{\e,k}(A,y)=f \mbox{ in }\Omega,\\
 \label{3a.3}
y\in H_0^1(\Omega),\;A\in\mathfrak{A}_{ad},\\
\label{3a.4}
z+BF_{\e,k}(y,z)=0 \mbox{ in }\Omega.
\end{gather}
Here,
$\mathfrak{A}_{ad}$ is defined in \eqref{1.2},
$k\in \mathbb{N}$, $\e$ is a small parameter, which varies within a
strictly decreasing sequence of positive numbers converging to $0$
and
\begin{gather}
\label{3a.5}
\mathcal{A}_{\e,k}(A,y)=\mathrm{div} \left(\left[\e+\mathcal{F}_k\big(|A^\frac{1}{2}\nabla y|^2\big)\right]^{\frac{p-2}{2}}A\nabla y\right),\\
\label{3aa.5}
F_{\e,k}(y,z)=\left[\e+\mathcal{F}_k(| y|^2)\right]^{\frac{p-2}{2}} y+\left[\e+\mathcal{F}_k(| z|^2)\right]^{\frac{p-2}{2}} z
\end{gather}
where $\mathcal{F}_k:\mathbb{R}_{+}\to \mathbb{R}_{+}$ is a non-decreasing $C^1(\mathbb{R}_{+})$-function such that
\begin{equation*}
\begin{aligned}
\mathcal{F}_k(t)=t,\ \text{ if }\ t\in\left[0,k^2\right],\quad \mathcal{F}_k(t)=k^2+1,\ \text{ if }\ t>k^2+1,\quad\text{and}\\
\quad t\le \mathcal{F}_k(t)\le t+\delta,\ \text{ if }\ k^2\le t< k^2+1\quad\text{ for some }\ \delta\in (0,1).
\end{aligned}
\end{equation*}
The main goal of this section is to show that, for each $\e>0$ and $k\in \mathbb{N}$, the approximating optimal control problem \eqref{3a.1}--\eqref{3a.4} is well posed and its solutions can be considered as a reasonable approximation of optimal pairs to the original problem \eqref{3.1*}--\eqref{3.1c*}. To begin with, we establish a few auxiliary results concerning monotonicity and growth conditions for the regularized anisotropic $p$-Laplacian $\mathcal{A}_{\e,k}$ and $F_{\e,k}$ (see for comparison \cite{KuMa15}).
\begin{remark}\label{Rem 4.1}It is clear that the effect of such perturbations of $\mathcal{A}(A,y)$ and $F$ is their regularization around critical points and points where $|A^\frac{1}{2}\nabla y(x)|$, $y$ and $z$  become unbounded. In particular, if $y\in W^{1,p}_0(\Omega)$ and $\Omega^1_k(A,y):=\left\{x\in\Omega\ :\ |A^\frac{1}{2}\nabla y(x)|>\sqrt{k^2+1}\right\}$, then the following chain of inequalities
\begin{align*}
&|\Omega^1_k(A,y)|:=\int_{\Omega^1_k(A,y)} 1\,dx\le \frac{1}{\sqrt{k^2+1}}\int_{\Omega^1_k(A,y)}|A^\frac{1}{2}\nabla y(x)|\,dx\\
&\le \frac{|\Omega^1_k(A,y)|^{\frac{1}{q}}}{\sqrt{k^2+1}} \left(\int_{\Omega} |A^\frac{1}{2}\nabla y|^pdx\right)^{\frac{1}{p}}\le \frac{\|\xi_2\|_{L^\infty(\Omega)}\|y\|_{W^{1,p}_0(\Omega)}}{\sqrt{k^2+1}} |\Omega^1_{k}(A,y)|^{\frac{p-1}{p}}
\end{align*}
shows that the Lebesgue measure of the set $\Omega^1_k(A,y)$ satisfies the estimate
\begin{equation}
\label{3a.5.1***}
|\Omega^1_k(A,y)|\le \left(\frac{\|\xi_2\|_{L^\infty(\Omega)}}{\sqrt{k^2+1}}\right)^p\|y\|^p_{W^{1,p}_0(\Omega)}\le \|\xi_2\|^p_{L^\infty(\Omega)}\|y\|^p_{W^{1,p}_0(\Omega)}k^{-p},
\end{equation}
for any element $y\in W^{1,p}_0(\Omega)$.
 For
$\Omega_k^2(y) :=\left\{x\in\Omega :  |y(x)|>\sqrt{k^2+1}\right\}$ and $\Omega_k^3(z) :=\left\{x\in\Omega : |z(x)|>\sqrt{k^2+1}\right\}$ we obviously get similar estimates for all $y\in W^{1,p}_0(\Omega)$ and $z\in L^p(\Omega)$
\begin{equation}\label{6.18}
|\Omega^2_k(y)|\le  \|y\|^p_{W^{1,p}_0(\Omega)}k^{-p},\quad
|\Omega^3_k(z)|\le  \|z\|^p_{L^p(\Omega)}k^{-p},
\end{equation}
which mean that approximations $\mathcal{F}_k(|A^\frac{1}{2}\nabla y|^2)$, $\mathcal{F}_k(|y|^2)$, $\mathcal{F}_k(|z|^2)$ are essential on sets with  small Lebesgue measure.
\end{remark}

\begin{proposition}
\label{Prop 3a.2}
For every $A\in \mathfrak{A}_{ad}$, $k\in \mathbb{N}$, and $\e>0$, the operator $\mathcal{A}_{\e,k}
(A,\cdot):H^1_0(\Omega)\to H^{-1}(\Omega)$ is bounded, strictly monotone, coercive (in the sense of relation \eqref{1.5}) and semi-continuous.
\end{proposition}
\begin{proof} The proof is given in Appendix.
\end{proof}

\begin{proposition}\label{prop 3.4}
For every $k\in\mathbb{N}$ and $\e>0$ the operator $F_{\e,k}: H_0^1(\Omega)\times L^2(\Omega)\to L^2(\Omega)$ is bounded, $F_{\e,k}(y,\cdot):L^2(\Omega)\to L^2(\Omega)$ is strictly monotone and  radially continuous for every $y\in H_0^1(\Omega)$, and $F_{\e,k}(\cdot,z):H_0^1(\Omega)\to L^2(\Omega)$ is compact in the following sense: if $y_n\rightharpoonup y_0$ in $H_0^1(\Omega)$, then  $F_{\e,k}(y_n,z)\to F_{\e,k}(y_0,z)$ strongly in $L^2(\Omega)$ as $n\to\infty$.
\end{proposition}
\begin{proof} The proof is given in Appendix.
\end{proof}

  Using above results we arrive at the following assertion.
\begin{proposition}\label{prop 1.15*}
 The set of admissible solutions to problem \eqref{3a.1}--\eqref{3a.4}
\begin{multline*}
\Xi_{\e,k}=\big\{(A,y,z)\in L^\infty(\Omega;\mathbb{S}^N)\times H_0^1(\Omega)\times L^2(\Omega)|\\ (A,y,z) \mbox{ are related by \eqref{3a.2}--\eqref{3a.4}}\big\}
\end{multline*}
is nonempty for every $f\in L^2(\Omega)$.
\end{proposition}
\begin{proof}
Properties of the operator $\mathcal{A}_{\e,k}(A,y)$ given by Proposition \ref{Prop 3a.2} imply, that for every fixed $\e>0$ and $k\in\mathbb{N}$ boundary value problem \eqref{3a.2}--\eqref{3a.3} admits a unique weak solution $y_A=y(A)\in H_0^1(\Omega)$ for every $A\in \mathfrak{A}_{ad}$ and $f\in L^2(\Omega)$. Moreover, the following estimate takes place
 \begin{multline*}
\e^{\frac{p-2}{2}}\alpha^2\|y\|^2_{H_0^1(\Omega)}\stackrel{\text{by \eqref{3.4.new}}}{\le}\langle \mathcal{A}_{\e,k} y,y\rangle_{H^{-1}(\Omega);H_0^1(\Omega)}\\
=\langle f,y\rangle_{H^{-1}(\Omega);H_0^1(\Omega)}\le\|f\|_{L^2(\Omega)}\|y\|_{H_0^1(\Omega)}.
 \end{multline*}
 Hence, we have
 $ \sup_{A\in \mathfrak{A}_{ad}}\|y_A\|_{H_0^1(\Omega)}\le\ds\frac{ \e^{\frac{2-p}{2}}}{\alpha^{2}}\|f\|_{L^2(\Omega)}$ .
And what is more, there exists $\lambda>0$ such that for any  $y\in H_0^1(\Omega): \|y\|_{H_0^1(\Omega)}\le \e^{\frac{2-p}{2}}\alpha^{-2}\|f\|_{L^2(\Omega)}$ and all $\|z\|_{L^2(\Omega)}>\lambda$ the inequality
$$\langle F_{\e,k}(y,z),z\rangle_{L^2(\Omega);L^2(\Omega)}\ge 0$$ holds true.
Since the operator $B$, given by \eqref{operator_B},  maps $L^2(\Omega)\to L^2(\Omega)$, it follows from Theorem \ref{Th 1.1*} that the set
$$
\mathcal{H}_{\e,k}(y_A)=\{z\in L^2(\Omega):\;z+BF_{\e,k}(y_A,z)=0\ \text{in the sense of distributions }\}
$$
is non-empty and weakly compact.
\end{proof}
\begin{definition}\label{def_tau1}
 We say that a sequence of triplets
$\{(A_k,y_k,z_k)\}_{k\in\mathbb{N}}$ from the space $L^\infty(\Omega;\mathbb{S}^{N})\times H^1_0(\Omega)\times L^2(\Omega)$
$\tau_1$-converges to a triplet $(A_0,y_0,z_0)$ if $A_k^{\frac{1}{2}}\stackrel{\ast}{\rightharpoonup} A_0^{\frac{1}{2}}$ in $BV(\Omega;\mathbb{S}^N)$,
$y_k\rightharpoonup y_0$ in $H_0^{1}(\Omega)$ and $z_k\rightharpoonup z_0$ in $L^2(\Omega)$.
\end{definition}

By analogy with Theorem \ref{Th 2.8} it is easy to show, that the set of admissible triplets $\Xi_{\e,k}$ to the optimal control problem  \eqref{3a.1}--\eqref{3a.4} is sequentially closed and compact with respect to $\tau_1$-topology of the set $L^\infty(\Omega;\mathbb{S}^N)\times H_0^{1}(\Omega)\times L^2(\Omega)$.
 We conclude the section with the following result.
 \begin{theorem}
\label{Th 3a.11} For every $\e>0$ and every integer $k\in \mathbb{N}$, the optimal control problem \eqref{3a.1}--\eqref{3a.4} is solvable, i.e. there exists a triplet $(A^{opt}_{\e,k},y^{opt}_{\e,k},z^{opt}_{\e,k})\in\Xi_{\e,k}$ such that
\[
I_{\e,k}(A^{opt}_{\e,k},y^{opt}_{\e,k},z^{opt}_{\e,k})=\inf_{(A,y,z)\in\Xi_{\e,k}}I_{\e,k}(A,y,z).
\]
\end{theorem}
\begin{proof}
Since the cost functional in \eqref{3a.1} is bounded from below and the set of admissible solutions $\Xi_{\e,k}$ is nonempty, it follows that there exists a minimizing sequence $\{(A_n,y_n,z_n)\}_{n\in \mathbb{N}}\subset \Xi_{\e,k}$ such that
$$
\lim_{n\to\infty}I_{\e,k}(A_n,y_n,z_n)=\inf_{(A,y,z)\in \Xi_{\e,k}}I(A,y,z).
$$
Hence, there exists a constant $C>0$ such that
$$
\|z_n\|_{L^2(\Omega)}\le \|z_n-z_d\|_{L^2(\Omega)}+\|z_d\|_{L^2(\Omega)}\le C.
$$
Moreover, in view of definition of the set $\mathfrak{A}_{ad}$,  we have
\begin{align*}
&\sup_{n\in \mathbb{N}}\left[ \|A^\frac{1}{2}_n\|_{BV(\Omega;\mathbb{S}^N)} +\|y_n\|_{H^1_0(\Omega)}+\|z_n\|_{L^2(\Omega)}\right]\\&\le \sqrt{N}\|\xi_2\|_{L^1(\Omega)}+\gamma+ \ds\frac{\e^{\frac{2-p}{2}}}{\alpha^{2}}\|f\|_{L^2(\Omega)}+C.
\end{align*}
Hence, there exists a subsequence $\left\{(A_{n_i},y_{n_i},z_{n_i})\right\}_{i\in \mathbb{N}}$ and a triplet $(A,y,z)\in L^\infty(\Omega)\times H^1_0(\Omega)\times L^2(\Omega)$ such that
\begin{align}
\notag
y_{n_i}\rightharpoonup y\ \text{ in $H^1_0(\Omega)$},&\quad y_{n_i}\rightarrow y\ \text{ in $L^2(\Omega)$},\quad z_{n_i}\rightharpoonup z\ \text{ in $L^2(\Omega)$},\\\label{3a.9.1}
A^\frac{1}{2}_{n_i}\rightarrow A^\frac{1}{2}\ \text{ in $L^1(\Omega)$},&\quad
A^\frac{1}{2}_{n_i}\rightarrow A^\frac{1}{2}\ \text{ almost everywhere in }\ \Omega,\\\notag
\gamma\ge \liminf_{i \to \infty}\int_\Omega& |D A^\frac{1}{2}_{n_i}|  \ge
\int_\Omega|D A^\frac{1}{2}|.
\end{align} In view of $\tau_1$-closedness of the set $\Xi_{\e,k}$,  we have $(A,y,z)\in\Xi_{\e,k}$. It remains to make  use of the lower semicontinuity of the cost functional with respect to the $\tau_1$-convergence
\begin{multline*}
I_{\e,k}(A,y,z)\le\liminf_{i\to\infty}I_{\e,k}(A_{n_i},y_{n_i},z_{n_i})\\
=\lim_{n\to\infty}I_{\e,k}(A_n,y_n,z_n)=\inf_{(A, y,z)\in \Xi_{\e,k}}I_{\e,k}(A,y,z).
\end{multline*}
The proof is complete.
\end{proof}

\section{Asymptotic Analysis of the Approximating OCP \eqref{3a.1}--\eqref{3a.4}}
\label{Sec 4a}

Our main intention in this section is to show that some optimal solutions to the original OCP \eqref{3.1*}--\eqref{3.1c*} can be attained (in certain sense) by optimal solutions to the approximating problems \eqref{3a.1}--\eqref{3a.4}.
With that in mind, we make use of the concept of variational convergence of constrained minimization problems (see \cite{KogutLeugering2011}). In order to study the asymptotic behaviour of a family of OCPs \eqref{3a.1}--\eqref{3a.4},
the passage to the limit in relations \eqref{3a.1}--\eqref{3a.4} as $\e\to 0$ and $k\to\infty$ has to be realized.
The expression \textquotedblleft passing to the limit\textquotedblright\ means that we have to find a kind
of \textquotedblleft limit cost functional\textquotedblright\ $I$ and \textquotedblleft limit set of constraints\textquotedblright\ $\Xi$
with a clearly defined structure such that the limit object
$\left<\inf_{(A,y,z)\in\Xi} I(A,y,z)\right>$ to the family \eqref{3a.1}--\eqref{3a.4} could be interpreted
as some~OCP.

Further we use the folowing notation
\begin{equation}\label{6.1}
\|y\|_{A,\e,k}=\left(\int_\Omega(\e+\mathcal{F}(|A^\frac{1}{2}\nabla y|^2))^\frac{p-2}{2}|A^\frac{1}{2}\nabla y|^2\,dx\right)^{1/p}.
\end{equation}
\begin{proposition}\label{prop 6.1} Let $A\in \mathfrak{A}_{ad}$, $k\in \mathbb{N}$, and $\e>0$ be given. Then, for arbitrary $g\in L^2(\Omega)$ and $y\in H^1_0(\Omega)$, we have
\begin{equation}
\label{3a.6.1}
\left|\int_\Omega g y\,dx\right|\le
C_\Omega\alpha^{-1}\|g\|_{L^2(\Omega)}\left(|\Omega|^\frac{p-2}{2p}\|y\|_{A,\e,k}+k^\frac{2-p}{2}\|y\|^{\frac{p}{2}}_{A,\e,k}\right).
\end{equation}
\end{proposition}
\begin{proof}The proof is given in Appendix.
\end{proof}
\begin{remark}
\label{Rem 3a.1.1}
For any fixed admissible control $A\in\mathfrak{A}_{ad}$ and an arbitrary element
$y^\ast\in H^1_0(\Omega)$ such that  $\|y^\ast\|_{A,\e,k}\le C<+\infty$ with a constant $C>0$ independent of $\e>0$ and $k\in \mathbb{N}$ for the set
$\Omega_k(A,y^\ast):=\{x\in\Omega: \;|A^\frac{1}{2}(x)\nabla y^\ast(x)|\ge\sqrt{k^2+1}\}$ we have
\begin{multline*}
|\Omega_k(A,y^\ast)| :=\int_{\Omega_k(A,y^\ast)} 1\,dx\le \frac{1}{\sqrt{k^2+1}}\int_{\Omega_k(A,y^\ast)}|A^\frac{1}{2}\nabla y^\ast(x)|\,dx\\
\le \frac{|\Omega_k(A,y^\ast)|^{\frac{1}{2}}}{k} \left(\int_{\Omega_k(A,y^\ast)} |A^\frac{1}{2}\nabla y^\ast|^2\,dx\right)^{\frac{1}{2}}\\
=\frac{|\Omega_k(A,y^\ast)|^{\frac{1}{2}}}{k(\e+k^2+1)^{\frac{p-2}{4}} }\left(\int_{\Omega_k(A,y^\ast)}\left(\e+\mathcal{F}_k(|A^\frac{1}{2}\nabla y^\ast|^2)\right)^{\frac{p-2}{2}} |A^\frac{1}{2}\nabla y^\ast|^2\,dx\right)^{\frac{1}{2}}\\
\le \frac{1}{k^\frac{p}{2}}\, |\Omega_k(A,y^\ast)|^{\frac{1}{2}}\, \|y^\ast\|^{\frac{p}{2}}_{A,\e,k}.
\end{multline*}
Hence, the Lebesgue measure of the set $\Omega_k(A,y^\ast)$ satisfies the estimate
\begin{equation}
\label{3a.6.0}
|\Omega_k(A,y^\ast)|\le \frac{\|y^\ast\|^p_{A,\e,k}}{k^p}\le \frac{C}{k^p},\quad\forall\,y^\ast\in H^1_0(\Omega)\ :\ \|y^\ast\|^p_{A,\e,k}\le C.
\end{equation}
\end{remark}
\begin{theorem}\label{teor 6.2}
For every $A\in \mathfrak{A}_{ad}$ and every $f\in L^2(\Omega)$ the sequence of weak solutions $\{y_{\e,k}=y_{\e,k}(A,f)\}_{\stackrel{\e>0}{k\in\mathbb{N}}}$ to boundary value problem \eqref{3a.2}--\eqref{3a.3} is uniformly bounded in $H_0^1(\Omega)$.
\end{theorem}
\begin{proof}
Using notation \eqref{6.1} and Proposition \ref{prop 6.1}, from \eqref{3a.2} we get
\begin{equation}\label{6.2}
\|y\|^p_{A,\e,k}=\int_\Omega fy_{\e,k}\,dx\le \frac{C_\Omega}{\alpha}\|f\|_{L^2(\Omega)}\left(|\Omega|^\frac{p-2}{2p}\|y\|_{A,\e,k}+k^\frac{2-p}{2}\|y\|^{\frac{p}{2}}_{A,\e,k}\right).
\end{equation}
Since $\|y_{\e,k}\|^{\frac{p}{2}-1}_{A,\e,k}\le \|y_{\e,k}\|^{p-1}_{A,\e,k}$ for $ \|y_{\e,k}\|_{A,\e,k}\ge 1$, it follows from \eqref{6.2} that
\begin{equation}\label{6.3}
\limsup_{\stackrel{\e\to 0}{k\to\infty}}\|y_{\e,k}\|^{p-1}_{A,\e,k}\le \limsup_{\stackrel{\e\to 0}{k\to\infty}} \frac{C_\Omega\alpha^{-1}\|f\|_{L^2(\Omega)}|\Omega|^\frac{p-2}{2p}}{1-k^\frac{2-p}{2}C_\Omega\alpha^{-1}\|f\|_{L^2(\Omega)}}=\frac{C_\Omega}{\alpha}\|f\|_{L^2(\Omega)}|\Omega|^\frac{p-2}{2p}.
\end{equation}
Hence, there exist $C>0$, $\e_0>0$ and $k_0>1$ such that $\sup_{\stackrel{\e < \e_0}{ k > k_0}}\|y_{\e,k}\|_{A,\e,k}\le C$ and the required assertion immediately follows from the estimate (see the proof of Proposition \ref{prop 6.1})
\begin{multline}\label{6.4*}
\|y_{\e,k}\|_{H_0^1(\Omega)}\le\alpha^{-1}\left(\int_\Omega |A^\frac{1}{2}\nabla y|^2\,dx\right)^\frac{1}{2}\\\le \alpha^{-1}(|\Omega|^\frac{p-2}{2p}\|y_{\e,k}\|_{A,\e,k}+\|y_{\e,k}\|^\frac{p}{2}_{A,\e,k}).
\end{multline}
\end{proof}

The following results are crucial for our further analysis.
\begin{theorem}
\label{Th 4a.3} Let $\{A_{\e,k}\}_{{\e>0} \atop {k\in \mathbb{N}}}\subset \mathfrak{A}_{ad}$ be a given sequence of admissible controls.
Let $\left\{y_{\e,k}=y_{\e,k}(A_{\e,k})\right\}_{{\e>0} \atop {k\in \mathbb{N}}}$ be a sequence of correspondent solutions to problem \eqref{3a.2}--\eqref{3a.3}.  Then
each cluster point $y$ of the sequence $\left\{y_{\e,k}\in H^1_0(\Omega)\right\}_{{\e>0} \atop {k\in \mathbb{N}}}$ with respect to the weak convergence in $H^1_0(\Omega)$, satisfies: $y\in W^{1,p}_0(\Omega)$.
\end{theorem}
\begin{proof}
To establish this property, we suppose, due to Theorem \ref{teor 6.2},  that there exists a subsequence $\left\{y_{\e_i,k_i}\right\}_{i\in \mathbb{N}}$ of $\left\{y_{\e,k}\right\}_{{\e>0} \atop {k\in \mathbb{N}}}$ (here, $\e_i\to 0$ and $k_i\to\infty$ as $i\to \infty$) and a
distribution $y\in H^1_0(\Omega)$ such that $y_{\e_i,k_i}\rightharpoonup y$ in $H^1_0(\Omega)$ as $i\to\infty$. Further, we
fix an index $i\in \mathbb{N}$ and associate  it with the following set
\begin{gather}
 \label{4a.4}
B_{i}:=\bigcup_{j=i}^\infty\Omega_{k_j}(A_{\e_j,k_j},y_{\e_j,k_j}),\\\notag\text{where }
\Omega_{k_j}(A_{\e_j,k_j},y_{\e_j,k_j}):=\left\{x\in\Omega\,:\, |A^\frac{1}{2}_{\e_j,k_j}\nabla y_{\e_j,k_j}|>\sqrt{k_j^2+1}\right\}.
\end{gather}

Due to estimates \eqref{3a.6.0}, we see that
\begin{gather*}
|B_{i}|\le \displaystyle\frac{\sum_{j=i}^\infty  \displaystyle\frac{1}{k_j^2}\|y_{\e_j,k_j}\|^p_{\e_j,k_j,u_{\e_j,k_j}}}{\alpha^2}\le
 \frac{\sup\limits_{j\in \mathbb{N}} \|y_{\e_j,k_j}\|^p_{\e_j,k_j,u_{\e_j,k_j}}}{\alpha^2}\sum_{j=i}^\infty  \frac{1}{k_j^2}<+\infty,
\end{gather*}
and, therefore,
\begin{equation}
\label{4a.5}
\lim_{i\to\infty} |B_{i}|=0.
\end{equation}
Using the fact that
\begin{align}
\notag
&\int_{\Omega\setminus B_i} |\nabla y_{\e_j,k_j}|^pdx\le \frac{1}{\alpha^p}\int_{\Omega\setminus B_i} \left(\e_j+|A^\frac{1}{2}_{\e_j,k_j}\nabla y_{\e_j,k_j}|^2\right)^{\frac{p-2}{2}}|A^\frac{1}{2}_{\e_j,k_j}\nabla y_{\e_j,k_j}|^2dx\\\label{4a.6}&
= \alpha^{-p}\int_{\Omega\setminus B_i} \left(\e_j+\mathcal{F}_{k_j}(|A^\frac{1}{2}_{\e_j,k_j}\nabla y_{\e_j,k_j}|^2)\right)^{\frac{p-2}{2}}|A^\frac{1}{2}_{\e_j,k_j}\nabla y_{\e_j,k_j}|^2\,dx,\,\forall\,j\ge i,
\end{align}
and \eqref{6.3},
we have that sequence $\{\nabla y_{\e_j,k_j}\}$ is bounded in $L^p(\Omega \setminus B_i)^N$. Since, $\nabla y_{\e_j,k_j} \rightharpoonup \nabla y$ in $L^2(\Omega)^N$, we infer that  $\nabla y_{\e_j,k_j} \rightharpoonup \nabla y$ in $L^p(\Omega)^N$ as well. Hence,
\begin{gather*}
\int_{\Omega} |\nabla y|^p\,dx \stackrel{\text{by \eqref{4a.5}}}{=} \lim_{i\to\infty}\int_{\Omega\setminus B_i} |\nabla y|^p\,dx
\le \lim_{i\to\infty}\liminf_{{j\to\infty}\atop{j\ge i}}\int_{\Omega\setminus B_i} |\nabla y_{\e_j,k_j}|^p\,dx\\
\stackrel{\text{by \eqref{4a.6}}}{\le}\,\alpha^{-p}\lim_{i\to\infty}\liminf_{{j\to\infty}\atop{j\ge i}} \sup_{{\e>0} \atop {k\in \mathbb{N}}}\|y_{\e,k}\|^p_{A_{\e,k},\e,k}<+\infty.
\end{gather*}
Thus, $y\in W^{1,p}_0(\Omega)$ and the proof is complete.
\end{proof}
\begin{proposition}
\label{Th 4a.4} Let $\{(A_{\e,k},y_{\e,k},z_{\e,k})\subset \Xi_{\e,k}\}_{{\e>0} \atop {k\in \mathbb{N}}} $ be a given sequence of admissible triplets to problem \eqref{3a.1}--\eqref{3a.4}.
Then the sequence $\left\{z_{\e,k}\in L^2(\Omega)\right\}_{{\e>0} \atop {k\in \mathbb{N}}}$ is bounded in $L^p(\Omega)$.
\end{proposition}
\begin{proof}
First, we show boundedness of the sequence $\left\{z_{\e,k} \right\}_{{\e>0} \atop {k\in \mathbb{N}}}$ in $L^2(\Omega)$. We are going to find $\lambda_{\e,k}\in\mathbb{R}$ such that for all $\|z\|\ge\lambda_{\e,k}$ and $y_{\e,k}\in Y_0=\{y\in H_0^1(\Omega)|\; y\mbox{ satisfies \eqref{6.3}-\eqref{6.4*}}\}$ the inequality  $\langle F_{\e,k}(y_{\e,k},z_{\e,k}),z_{\e,k}\rangle_{L^2(\Omega);L^2(\Omega)}\ge 0$ holds (see Theorem \ref{Th 2.8}).
 We have
\begin{align*}
\langle F_{\e,k}&(y_{\e,k},z_{\e,k}),z_{\e,k}\rangle_{L^2(\Omega);L^2(\Omega)}\\&=\int_\Omega\left((\e+\mathcal{F}_k(y_{\e,k}^2))^\frac{p-2}{2}y_{\e,k}z_{\e,k}
+(\e+\mathcal{F}_k(z_{\e,k}^2))^\frac{p-2}{2}z_{\e,k}^2\right)dx\\&
\ge \e\left(\int_{\Omega} (y_{\e,k}z_{\e,k}+z_{\e,k}^2)\,dx\right)\ge\e \left(\|z_{\e,k}\|^2_{L^2(\Omega)}-\|y_{\e,k}\|_{L^2(\Omega)}\|z_{\e,k}\|_{L^2(\Omega)}\right)\\&=
\e\|z_{\e,k}\|_{L^2(\Omega)}(\|z_{\e,k}\|_{L^2(\Omega)}-\|y_{\e,k}\|_{L^2(\Omega)})\ge 0,
\end{align*}
which implies $\|z_{\e,k}\|_{L^2(\Omega)}\ge\|y_{\e,k}\|_{L^2(\Omega)}=\lambda_{\e,k}$ and vice versa.

As $z_{\e,k}$ is a solution of Hammerstein equation \eqref{3a.4}, the following estimate takes place
$$
\|z_{\e,k}\|_{L^2(\Omega)}\le\lambda_{\e,k}=\|y_{\e,k}\|_{L^2(\Omega)}\le\|y_{\e,k}\|_{H_0^1(\Omega)}\le C,\mbox{ since }y\in Y_0.
$$
Now we show that $z_{\e,k}\in L^p(\Omega)$. Indeed, $B$ maps $L^q(\Omega)\to L^p(\Omega)$ (see \eqref{operator_B}) and $F_{\e,k}(y,z)\in L^2(\Omega)\subset L^q(\Omega)$, we immediately obtain that any solution of the equation
$z+B{F}_{\e,k}(y,z)=0$ belongs to $L^p(\Omega)$. Moreover, since $B$ is also linear continuous operator, mapping $L^2(\Omega)$ to $L^2(\Omega)$, it cannot map unbounded sets into bounded, hence there exists a constant $C_1>0$ such that $\|F_{\e,k}(y_{\e,k},z_{\e,k})\|_{L^2(\Omega)} \le C_1$. By H\"{o}lder inequality
$$
\|F_{\e,k}(y_{\e,k},z_{\e,k})\|_{L^q(\Omega)}\le |\Omega|^\frac{2-q}{2q}\|F_{\e,k}(y_{\e,k},z_{\e,k})\|_{L^2(\Omega)}\le\widetilde{C}_1.
$$
As a result, the boundedness of $B$ implies existence of a constant $C_2$ such that $\|z_{\e,k}\|_{L^p(\Omega)}\le C_2$.
\end{proof}
\begin{proposition}\label{prop 4.1}
Let $\{y_n\}_{n\in \mathbb{N}}\subset L^2(\Omega)$ be a given sequence such that
$$
y_n\to y_0 \mbox{ in }L^2(\Omega),\; \|\left(\e_n+\mathcal{F}_{k_n}(y_n^2)\right)^\frac{p-2}{2}y_n\|_{L^2(\Omega)}\le C \mbox{ and }y_0\in L^p(\Omega).
$$
Then, within a subsequence,
\begin{equation}\label{6.5}
\left(\e_n+\mathcal{F}_{k_n}(y_n^2)\right)^\frac{p-2}{2}y_n\rightarrow |y_0|^{p-2}y_0\mbox{ strongly in } L^q(\Omega).
\end{equation}
\end{proposition}
\begin{proof}
The proof is given in Appendix.
\end{proof}
\begin{proposition}\label{cor 6.1}
Let $\e_n\to 0$ and $k_n\to \infty$ as $n\to \infty$ and $\{(A_n,y_n,z_n)\subset \Xi_{\e_n,k_n}\}_{n\in \mathbb{N}}$ be a sequence of admissible triplets to problem \eqref{3a.1}--\eqref{3a.4}, such that
\begin{equation*}
A_n^{\frac{1}{2}}\stackrel{\ast}{\rightharpoonup} A_0^{\frac{1}{2}} \text{ in $BV(\Omega;\mathbb{S}^N)$},\quad
y_n\rightharpoonup y_0\ \text{ in $H^1_0(\Omega)$},\quad z_n\rightharpoonup z_0\text{ in $L^2(\Omega)$}.
 \end{equation*}
Then $F_{\e_n,k_n}(y_n,z_n)\rightharpoonup F(y_0,z_0)=|y_0|^{p-2}y_0+|z_0|^{p-2}z_0$   in $L^q(\Omega)$, where
$$
z_0+BF(y_0,z_0)=0,
$$
and $z_n\to z_0$ strongly in $L^p(\Omega)$.
\end{proposition}
\begin{proof}
\textit{Step 1.} Here we show, that
\begin{equation}\label{6.8a}
z_n+B\left((\e_n+\mathcal{F}_{k_n}(z_n^2))^\frac{p-2}{2}z_n\right)\to z_0+B |z_0|^{p-2}z_0\mbox{ strongly in }L^p(\Omega).
\end{equation}
in view of equation \eqref{3a.4} we have $z_n=-BF_{\e_n,k_n}(y_n,z_n)$. Since sequence $\{z_n\}_{n\in\mathbb{N}}$ is bounded in $L^p(\Omega)$ (see Proposition \ref{Th 4a.4}) then the following estimate
$\|(\e_n+\mathcal{F}_{k_n}(z_n^2))^\frac{p-2}{2}z_n\|_{L^q(\Omega)}\le C_1$ takes place. Indeed,
\begin{multline*}
\int_\Omega \left|(\e_n+\mathcal{F}_{k_n}(z_n^2))^\frac{p-2}{2}z_n\right|^\frac{p}{p-1}\,dx\le \int_\Omega (\e_n+z_n^2)^\frac{p(p-2)}{2(p-1)}|z_n|
^\frac{p}{p-1}\,dx\\\le \max\{2^{\frac{p(p-2)}{2(p-1)}-1};1\}\int_\Omega \left(\e_n^\frac{p(p-2)}{2(p-1)}|z_n|^\frac{p}{p-1}+|z_n|^p\right)\,dx\le C_1,
\end{multline*}
because
$$
(\e_n+z_n^2)^r\le\left\{
\begin{aligned}
&2^{r-1}(\e_n^r+z_n^{2r}),&\forall\, r\ge 1\,\mbox{(by convexity of function}\,g(x)=x^r),\\
&\e_n^r+z_n^{2r},&\forall\, 0<r<1\qquad\qquad\quad\mbox{(see \cite[Teorem 27]{Hardy})}.
\end{aligned}
\right.
$$
Continuity of $B$ implies boundedness in $L^p(\Omega)$ of the left-hand side of equation
\begin{equation}\label{6.7}
z_n+B\left((\e_n+\mathcal{F}_{k_n}(z_n^2))^\frac{p-2}{2}z_n\right)=-B\left((\e_n+\mathcal{F}_{k_n}(y_n^2))^\frac{p-2}{2}y_n\right),\quad\forall\;n\in\mathbb{N}.
\end{equation}
Hence, the right-hand side is bounded in $L^p(\Omega)\subset L^2(\Omega)$ as well, and, therefore,
$$
\|B(\e_n+\mathcal{F}_{k_n}(y_n^2))^\frac{p-2}{2}y_n\|_{L^2(\Omega)}\le C_2\,\Longleftrightarrow\, \|(\e_n+\mathcal{F}_{k_n}(y_n^2))^\frac{p-2}{2}y_n\|_{L^2(\Omega)}\le C_3.
$$
 Using Proposition \ref{prop 4.1} and continuity of $B$, we obtain the strong convergence of the sequence in the right-hand side of equation \eqref{6.7}, i.e.
$$-B\left((\e_n+\mathcal{F}_{k_n}(y_n^2))^\frac{p-2}{2}y_n\right)\to -B|y_0|^{p-2}y_0\mbox{ strongly in }L^p(\Omega).$$
In particular, this fact leads to the strong convergence of the left-hand side, i.e.
\begin{equation}\label{6.9}
z_n+B\left((\e_n+\mathcal{F}_{k_n}(z_n^2))^\frac{p-2}{2}z_n\right)\to z_0+B\zeta\mbox{ strongly in }L^p(\Omega),
\end{equation}
where $\zeta\in L^q(\Omega)$ is a weak limit of $\{(\e_n+\mathcal{F}_{k_n}(z_n^2))^\frac{p-2}{2}z_n\}_{n\in\mathbb{N}}$ within a subsequence.

Now we are in a position to show, that $\xi=|z_0|^{p-2}z_0$.
Indeed, let  $w_n\in L^q(\Omega)$ is such that $z_n=B^\ast w_n$ $\forall\, n\in\mathbb{N}$. In view of properties of $B$, $w_n\rightharpoonup w_0$ weakly in $L^q(\Omega)$, $B^\ast w_0=z_0$. Now we multiply on $w_n$ both sides of equality \eqref{6.7}. We have
\begin{multline}\label{6.8}
\langle w_n,B^\ast w_n\rangle_{L^q(\Omega);L^p(\Omega)}+\langle(\e_n+\mathcal{F}_{k_n}((B^\ast w_n)^2))^\frac{p-2}{2}B^\ast w_n,B^\ast w_n\rangle_{L^q(\Omega);L^p(\Omega)}\\=-\langle(\e_n+\mathcal{F}_{k_n}(y_n)^2)^\frac{p-2}{2}y_n,B^\ast w_n\rangle_{L^q(\Omega);L^p(\Omega)}.
\end{multline}
If we put $B_n(x,\xi)=(\e_n+\mathcal{F}_{k_n}(\xi)^2)^\frac{p-2}{2}\xi$, then $\lim_{n\to\infty}B_n(x,\xi)=|\xi|^{p-2}\xi$ and Lemma \ref{Gikov1} implies
\begin{multline}\label{6.10}
\liminf_{n\to\infty}\Big(\langle w_n+(\e_n+\mathcal{F}_{k_n}((B^\ast w_n)^2))^\frac{p-2}{2}B^\ast w_n,B^\ast w_n\rangle_{L^p(\Omega);L^q(\Omega)}\Big)\\
\ge\liminf_{n\to\infty}\langle w_n,B^\ast w_n\rangle_{L^p(\Omega);L^q(\Omega)}\\+\liminf_{n\to\infty}\langle(\e_n+\mathcal{F}_{k_n}((B^\ast w_n)^2))^\frac{p-2}{2}B^\ast w_n,B^\ast w_n\rangle_{L^p(\Omega);L^q(\Omega)}\\
\ge \langle w_0,B^\ast w_0\rangle_{L^p(\Omega);L^q(\Omega)}+\langle \zeta,B^\ast w_0\rangle_{L^q(\Omega);L^p(\Omega)}=\langle z_0+B\zeta, w_0\rangle_{L^p(\Omega);L^q(\Omega)}.
\end{multline}
However, in view of \eqref{6.9}, all inequalities in \eqref{6.10} become equalities and Lemma \ref{Gikov1} implies
$\zeta=|z_0|^{p-2}z_0$ and, in particular,
after passing to the limit in \eqref{6.7}, we obtain the desired equation
$$z_0+B|z_0|^{p-2}z_0=-B|y_0|^{p-2}y_0\mbox{ or }z_0+BF(y_0,z_0)=0.$$

\textit{Step 2.} We are left to show the strong convergence of $\{z_n\}_{n\in\mathbb{N}}$ to element $z_0$ in $L^p(\Omega)$. Indeed, by \eqref{6.8a}
$$
\alpha _n:=z_n-z_0+B\left((\e_n+\mathcal{F}_{k_n}(z_n^2))^\frac{p-2}{2}z_n-|z_0|^{p-2}z_0\right)\to 0\, \mbox{strongly in}\,L^p(\Omega),
$$
and $\langle \alpha_n, w_n-w_0\rangle_{L^p(\Omega);L^q(\Omega)}\rightarrow 0$ as $n\to\infty$,
since $w_n-w_0\rightharpoonup 0$ weakly in $L^q(\Omega)$. Let us denote $\Omega_{k_n}(z_n):=\{x\in\Omega:\,z_n(x)>\sqrt{k_n^2+1}\}$ and,  by \eqref{6.18},
$|\Omega_{k_n}(z_n)|\le \|z_n\|^p_{L^p(\Omega)}k_n^{-p}\le Ck_n^{-p}$. Then the following chain of relations takes place
\begin{multline}\label{6.12}
\langle \alpha_n, w_n-w_0\rangle_{L^p(\Omega);L^q(\Omega)}=\langle z_n-z_0, w_n-w_0\rangle_{L^p(\Omega);L^q(\Omega)}\\+\left\langle B\left((\e_n+\mathcal{F}_{k_n}(z_n^2))^\frac{p-2}{2}z_n-|z_0|^{p-2}z_0\right), w_n-w_0\right\rangle_{L^p(\Omega);L^q(\Omega)}\\
=\langle B^\ast(w_n-w_0), w_n-w_0\rangle_{L^p(\Omega);L^q(\Omega)}\\+\langle (\e_n+\mathcal{F}_{k_n}(z_n^2))^\frac{p-2}{2}z_n-|z_0|^{p-2}z_0, z_n-z_0\rangle_{L^q(\Omega);L^p(\Omega)}\\
=\langle B^\ast(w_n-w_0), w_n-w_0\rangle_{L^p(\Omega);L^q(\Omega)} \\+\int_{\Omega\setminus \Omega_{k_n}(z_n)}\big(|z_n|^{p-2}z_n-|z_0|^{p-2}z_0\big)(z_n-z_0)\,dx\\
+\int_{\Omega\setminus \Omega_{k_n}(z_n)}\big((\e_n+z_n^2)^\frac{p-2}{2}z_n-|z_n|^{p-2}z_n\big)(z_n-z_0)\,dx\\+ \int_{\Omega_{k_n}(z_n)}\big((\e_n+\mathcal{F}_{k_n}(z_n^2))^\frac{p-2}{2}z_n-|z_0|^{p-2}z_0\big)( z_n-z_0)\,dx
\\=I^n_1+I^n_2+I^n_3+I^n_4\rightarrow 0.
\end{multline}
Let us show, that $I^n_j\to 0$ for $j=3,4$. Indeed, as for $I^n_4$, we easily get the desired result,  since$|\Omega_{k_n}(z_
n)|\le\frac{C}{k_n^p}\stackrel{n\to\infty}{\longrightarrow} 0$, as follows:
\begin{align}\notag
I^n_4&\le \int_{\Omega_{k_n}(z_n)}(\e_n+(z_n)^2)^\frac{p-2}{2}|z_n(z_n-z_0)|\,dx+\int_{\Omega_{k_n}(z_n)}|z_0|^{p-1}|z_n-z_0|\,dx\\\notag
&\le \int_{\Omega_{k_n}(z_n)}\left[\max\{2^\frac{p-4}{2},1\}\left(\e^\frac{p-2}{2}|z_n|+|z_n|^{p-1}\right)+|z_0|^{p-1}\right]|z_n-z_0|\,dx
\\\notag
&\le\left(\e^\frac{p-2}{2}\|z_n\|_{L^q(\Omega_{k_n}(z_n))}+\|z_n\|^{p-1}_{L^p(\Omega_{k_n}(z_n))} + \|z_0\|^{p-1}_{L^p(\Omega_{k_n}(z_n))}\right)\\\label{6.13}&\times\max\{2^\frac{p-4}{2},1\}\|z_n-z_0\|_{L^p(\Omega_{k_n}(z_n))}
  \to 0.
\end{align}
And by Lebesgue's theorem (see Lemma \ref{Th_1.9}), we have
\begin{equation}\label{6.12a}
|I^n_3|\le\int_{\Omega\setminus \Omega_{k_n}(z_n)}\left|(\e_n+z_n^2)^\frac{p-2}{2}-|z_n|^{p-2}\right|\cdot|z_n(z_n-z_0)|\,dx\rightarrow 0,
\end{equation}
as the integrand converges to zero a.e. in $\Omega\setminus \Omega_{k_n}(z_n)$ and the estimate
\begin{multline*}
|I^n_3|\le
\left(\max\{2^\frac{p-4}{2};1\}\e^\frac{p-2}{2}\|z_n\|_{L^q(\Omega)}+|2^\frac{p-4}{2}-1|\cdot\|z_n\|^{p-1}_{L^p(\Omega)}\right)
\\\times\|z_n-z_0\|_{L^p(\Omega)}\le C
\end{multline*}
provides its equi-integrability property.

Hence, combining \eqref{6.12},\eqref{6.13} and \eqref{6.12a}, we get $I^n_1+I^n_2\to 0$ as $n\to \infty$. Since $I^n_1\ge 0$ and $I^n_2\ge 0$ for all $n\in\mathbb{N}$, it follows that $I^n_1\to 0$ and $I^n_2\to 0$. However,
\begin{multline}\label{6.13a}
I^n_2=\int_{\Omega}\big(|z_n|^{p-2}z_n-|z_0|^{p-2}z_0\big)(z_n-z_0)\,dx\\
-\int_{\Omega_{k_n}(z_n)}\big(|z_n|^{p-2}z_n-|z_0|^{p-2}z_0\big)(z_n-z_0)\,dx=J_1^n-J^n_2\rightarrow 0,
\end{multline}
where similarly to \eqref{6.13}
\begin{multline}\label{6.13b}
|J_2|\le\big(\|z_n\|^{p-1}_{L^p(\Omega_{k_n}(z_n))}+\|z_0\|^{p-1}_{L^p(\Omega_{k_n}(z_n))}\big)\\\times
\big(\|z_n\|_{L^p(\Omega_{k_n}(z_n))}+\|z_0\|_{L^p(\Omega_{k_n}(z_n))}\big)\to 0\mbox{ since } |\Omega_{k_n}(z_
n)|\stackrel{n\to\infty}{\longrightarrow} 0.
\end{multline}
Hence, by well known inequality $(|a|^{p-2}a-|b|^{p-2}b)(a-b)\ge 2^{2-p}|a-b|^p$, taking into account \eqref{6.13a} and \eqref{6.13b},  we have
$$
\|z_n-z_0\|^p_{L^p(\Omega)}\le 2^{p-2}J_1^n\rightarrow 0\mbox{ as }n\to\infty.
$$
Thus $z_n\to z$ in $L^p(\Omega)$. The proof is complete.
\end{proof}

We are now in a position to show that optimal pairs to approximating OCP \eqref{3a.1}--\eqref{3a.4} lead in the limit to some optimal solutions to the original OCP \eqref{3.1*}--\eqref{3.1c*}. With that in mind we make use of the scheme of the direct variational convergence of OCPs \cite{KogutLeugering2011}. We begin with
the following definition for the convergence of constrained minimization problems.
\begin{definition}
\label{Def 3a.11} A problem $\left<\inf_{(A,y,z)\in\Xi} I(A,y,z)\right>$ is
the variational limit of the sequence $\left\{\left<\inf_{(A,y,z)\in\Xi_{\e,k}} I_{\e,k}(A,y,z)\right>; {{\e>0} \atop {k\in \mathbb{N}}}\right\}$ as $\e\to 0$ and $k\to\infty$
with respect to the $\tau_1$-convergence in $L^\infty(\Omega;\mathbb{S}^N)\times H^1_0(\Omega)\times L^2(\Omega)$, if the following conditions are satisfied:
\begin{enumerate}
\item[$\quad$(d)] If sequences $\left\{\e_n\right\}_{n\in \mathbb{N}}$, $\left\{k_n\right\}_{n\in \mathbb{N}}$, and
$\left\{(A_n,y_n,z_n)\right\}_{n\in \mathbb{N}}$ are such that $\e_n
\rightarrow 0$ and $k_n\to\infty$ as $n\rightarrow \infty$, $(A_n,y_n,z_n)\in
\Xi_{\e_n,k_n}$ $\forall\,n\in \mathbb{N}$, and
$(A_n,y_n,z_n)\stackrel{\tau_1}{\longrightarrow} (A,y,z)$ in $L^\infty(\Omega;\mathbb{S}^N)\times H^1_0(\Omega)\times L^2(\Omega)$ as follows
\begin{equation}
\label{3a.12.1}
\begin{split}
y_n\rightharpoonup y\ &\text{ in $H^1_0(\Omega)$},\quad y_n\rightarrow y\ \text{ in $L^2(\Omega)$},\quad z_n\rightharpoonup z\text{ in $L^2(\Omega)$}\\
&A_n^{\frac{1}{2}}\stackrel{\ast}{\rightharpoonup} A_0^{\frac{1}{2}} \text{ in $BV(\Omega;\mathbb{S}^N)$}, \quad A_n^{\frac{1}{2}}{\rightarrow} A_0^{\frac{1}{2}} \text{ in $L^1(\Omega;\mathbb{S}^N)$},
\end{split}
\end{equation}
then
\begin{equation}
\label{3a.12} (A,y,z)\in \Xi;\quad I(A,y,z)\le
\liminf_{n\to\infty}I_{\e_n,k_n}(A_n,y_n,z_n);
\end{equation}

\item[$\quad$(dd)] For every $(A,y,z) \in \Xi\subset L^\infty(\Omega;\mathbb{S}^N)\times W^{1,p}_0(\Omega)\times L^p(\Omega)$, there exists a sequence
$\left\{(A_{\e,k},y_{\e,k},z_{\e,k})\right\}_{{\e>0} \atop {k\in \mathbb{N}}}$ (called a $\Gamma$-realizing sequence) such that
\begin{gather}
\label{3a.13b}
(A_{\e,k},y_{\e,k},z_{\e,k})\in\Xi_{\e,k},\ \forall\,\e>0,\ \forall\, k\in \mathbb{N},\\
\label{3a.13a}
(A_{\e,k},y_{\e,k},z_{\e,k})\stackrel{\tau_1}{\longrightarrow}(A,y,z)\\
\label{3a.13c} I(A,y,z)\ge \limsup_{{\e\to 0} \atop {k\to\infty}}
I_{\e,k}(A_{\e,k},y_{\e,k},z_{\e,k}).
\end{gather}
\end{enumerate}
\end{definition}
Then the following result holds true \cite{KogutLeugering2011}.
\begin{theorem}
\label{Th 3a.12} Assume that the constrained minimization problem
\begin{equation}
\label{3a.14} \Big \langle \inf_{(A,y,z) \in \Xi} I(A,y,z)\Big
\rangle
\end{equation}
is the variational limit of sequence $\left\{\left<\inf_{(A,y,z)\in\Xi_{\e,k}} I_{\e,k}(A,y,z)\right>; {{\e>0} \atop {k\in \mathbb{N}}}\right\}$ as $\e\to 0$ and $k\to\infty$ in the sense of Definition \ref{Def 3a.11} and this problem has a nonempty set of solutions
$$
\Xi^{opt}:=\left\{(A^0,y^0,z^0)\in\Xi\ :\ I(A^0,y^0,z^0)=\inf_{(A,y,z) \in \Xi} I(A,y,z)\right\}.
$$
For every $\e>0$ and $k\in \mathbb{N}$, let
$(A^0_{\e,k},y^0_{\e,k},z^0_{\e,k})\in \Xi_{\e,k}$ be a minimizer of $I_{\e,k}$ on the
corresponding set $\Xi_{\e,k}$. If the sequence
$\{(A^0_{\e,k},y^0_{\e,k},z^0_{\e,k})\}_{{\e>0} \atop {k\in \mathbb{N}}}$ is bounded in $L^\infty(\Omega;\mathbb{S}^N)\times H_0^1(\Omega)\times L^2(\Omega)$, then
there exists a triplet $(A^0,y^0,z^0)\in \Xi$ such that
\begin{gather}
\label{3a.15} \begin{split}
y^0_{\e,k}&\rightharpoonup y^0\ \text{ in $H^1_0(\Omega)$},\quad y^0_{\e,k}\rightarrow y^0\ \text{ in $L^2(\Omega)$},\quad z^0_{\e,k}\rightharpoonup z^0  \text{ in $L^2(\Omega)$},\\
&A_{\e,k}^\frac{1}{2}\rightarrow (A^0)^\frac{1}{2}\ \text{ in $L^1(\Omega;\mathbb{S}^N)$},\quad A_{\e,k}^{\frac{1}{2}}\stackrel{\ast}{\rightharpoonup} (A^0)^{\frac{1}{2}} \text{ in $BV(\Omega;\mathbb{S}^N)$},
\end{split}\\\notag
\inf_{(A,y,z)\in\,\Xi}I(A,y,z)= I\left(A^0,y^0,z^0\right) =\lim_{{\e\to 0} \atop {k\to\infty}}
I_{\e,k}(A^0_{\e,k},y^0_{\e,k},z^0_{\e,k}) \\\label{3a.16}=\lim_{{\e\to 0} \atop {k\to\infty}}\inf_{(A,y,z)\in\,{\Xi}_{\e,k}}
{I}_{\e,k}(A,y,z).
\end{gather}
\end{theorem}
The main result of this section can be stated as follows.
\begin{theorem}
\label{Th 3a.17}
The optimal control problem \eqref{3.1*}--\eqref{3.1c*} is the
variational limit of the sequence \eqref{3a.1}--\eqref{3a.4} as $\e\to 0$ and $k\to\infty$.
\end{theorem}
\begin{proof}
To show, that all conditions of Definition \ref{Def 3a.11} hold true, we begin with the property $(d)$. Let $\left\{\e_n\right\}_{n\in \mathbb{N}}$, $\left\{k_n\right\}_{n\in \mathbb{N}}$, and
$\left\{(A_n,y_n,z_n)\right\}_{n\in \mathbb{N}}$ be sequences such that $\e_n
\rightarrow 0$ and $k_n\to\infty$ as $n\rightarrow \infty$, $(A_n,y_n,z_n)\in
\Xi_{\e_n,k_n}$ $\forall\,n\in \mathbb{N}$, and
$(A_n,y_n,z_n)\rightarrow (A,y,n)$ in the sense of relations \eqref{3a.12.1}.
We note that $y\in W^{1,p}_0(\Omega)$ by Theorem~\ref{Th 4a.3}.
Since the inequality \eqref{3a.12} is a direct consequence of semicontinuity of the cost functional $I$ with respect to $\tau_1$-convergence in $L^\infty(\Omega;\mathbb{S}^N)\times H^1_0(\Omega)\times L^2(\Omega)$, it remains to show that $(A,y,z)\in \Xi$. To this end, we note that the inclusion $A\in \mathfrak{A}_{ad}$ is guaranteed by the strong convergence $A^\frac{1}{2}_n\rightarrow A^\frac{1}{2}$ in $L^1(\Omega)$ and condition $A^\frac{1}{2}_n\in \mathfrak{A}_{ad}$ for all $n\in \mathbb{N}$. In order to show that $(A,y)$ is related by \eqref{3.1a*}, let us fix an arbitrary  function $\varphi\in C^\infty_0(\Omega)$ and pass to the limit in the Minty inequality (see Remark \ref{Rem 1.1})
\begin{equation}
\label{3a.18}
\int_\Omega (\e_n+\mathcal{F}_{k_n}(|A_n^\frac{1}{2}\nabla \varphi|^2))^{\frac{p-2}{2}}\left(A_n\nabla \varphi,\nabla\varphi-\nabla y_n\right)_{\mathbb{R}^N}\,dx \ge\int_\Omega {f}(\varphi- y_n)\,dx,
\end{equation}
as $n\to\infty$.  Taking into account that $A_n^\frac{1}{2}\nabla \varphi\to A^\frac{1}{2}\nabla \varphi$ strongly in $L^r(\Omega)^N$ any $r> 1$, we have (see for comparison Proposition \ref{prop 4.1})
$$
(\e_n+\mathcal{F}_{k_n}(|A_n^\frac{1}{2}\nabla \varphi|^2))^{\frac{p-2}{2}}A_n^\frac{1}{2}\nabla \varphi\rightarrow |A^\frac{1}{2}\nabla \varphi|^{p-2}A^\frac{1}{2}\nabla \varphi\ \text{  strongly in $L^q(\Omega)^N$},
$$
and making use of Lemma~\ref{Lemma 3.0a}, we get
\begin{align*}
\lim_{n\to\infty} &\int_\Omega (\e_n+\mathcal{F}_{k_n}(|A_n^\frac{1}{2}\nabla \varphi|^2))^{\frac{p-2}{2}}\left(A_n\nabla \varphi,\nabla\varphi\right)_{\mathbb{R}^N}\,dx\\=& \int_\Omega |A^\frac{1}{2}\nabla \varphi|^{p-2}\left(A\nabla \varphi,\nabla\varphi\right)_{\mathbb{R}^N}\,dx,\\
\lim_{n\to\infty} &\int_\Omega (\e_n+\mathcal{F}_{k_n}(|A_n^\frac{1}{2}\nabla \varphi|^2))^{\frac{p-2}{2}}\left(A_n\nabla \varphi,\nabla y_n\right)_{\mathbb{R}^N}\,dx\\= &\int_\Omega |A^\frac{1}{2}\nabla \varphi|^{p-2}\left(A\nabla \varphi,\nabla y\right)_{\mathbb{R}^N}\,dx.
\end{align*}
Upon passing to the limit in \eqref{3a.18} as $n\to\infty$, we arrive at  the relation
\begin{equation*}
\int_\Omega |A^\frac{1}{2}\nabla \varphi|^{p-2}\left(A\nabla \varphi,\nabla\varphi-\nabla y\right)_{\mathbb{R}^N}\,dx \ge\int_\Omega {f}(\varphi- y)\,dx,
\end{equation*}
which means that $y=y(A)\in W^{1,p}_0(\Omega)$ is a weak solution to the boundary value problem \eqref{3.1a*}--\eqref{3.1b*}. Making use of Proposition \ref{cor 6.1} we obtain that $z=z(A,y)$ is a solution of Hammerstein equation \eqref{3.1c*}, so $(A,y,z)\in \Xi$.

The next step is to prove  property (dd) of Definition~\ref{Def 3a.11}. Let $(A,y,z)\in\Xi$ be an arbitrary admissible pair to the original OCP \eqref{3.1*}--\eqref{3.1c*}. We construct a $\Gamma$-realizing sequence $\left\{(A_{\e,k},y_{\e,k},z_{\e,k})\right\}_{{\e>0} \atop {k\in \mathbb{N}}}$ as follows: $A_{\e,k}\equiv A$ for all $\e>0$ and $k\in \mathbb{N}$, and $y_{\e,k}$ is a corresponding weak solution to regularized BVP \eqref{3a.2}--\eqref{3a.3} under $A=A_{\e,k}$ and $z_{\e,k}$ is a solution of regularized Hammerstein equation \eqref{3a.4} under $y=y_{\e,k}$. Then, $(A_{\e,k},y_{\e,k},z_{\e,k})\in\Xi_{\e,k}$ for all $\e>0$ and $k\in \mathbb{N}$, and, as follows from  Theorem~\ref{teor 6.2} and Proposition~\ref{Th 4a.4}, this sequence is relatively compact with respect to the $\tau_1$-convergence in $L^\infty(\Omega;\mathbb{S}^N)\times H^1_0(\Omega)\times L^2(\Omega)$. Hence, applying the arguments of the previous step, we obtain: all cluster pairs of the sequence $\left\{(A_{\e,k},y_{\e,k},z_{\e,k})\right\}_{{\e>0} \atop {k\in \mathbb{N}}}$ with respect to the $\tau_1$-convergence in $L^\infty(\Omega;\mathbb{S}^N)\times H^1_0(\Omega)\times L^2(\Omega)$ are related by \eqref{3.1a*}--\eqref{3.1c*} and belong to $L^\infty(\Omega;\mathbb{S}^N)\times W^{1,p}_0(\Omega)\times L^p(\Omega)$. The boundary value problem \eqref{3.1a*}--\eqref{3.1b*} has a unique weak solution for each $A\in \mathfrak{A}_{ad}$, hence $y_{\e,k}\rightharpoonup y$ in $H_0^1(\Omega)$. Just as well, according to Remark \ref{unique}, the Hammerstein equation \eqref{3.1c*} has a unique solution for each $y\in W_0^{1,p}(\Omega)$, hence, $z_{\e,k}\rightharpoonup z$.
It remains to establish relation \eqref{3a.13c}, which obviously holds due to strong convergence $z_{\e,k}\to z$, given by Proposition \ref{cor 6.1}.
\end{proof}


\section*{Appendix}.

\textbf{Proof of Proposition \ref{Prop 3a.2}}

\textit{Boundedness.}
From the assumptions on $\mathcal{F}_k$ and the boundedness of $A$, we get
\begin{align*}
\|\mathcal{A}_{\e,k}\|&=\sup_{\|y\|_{H^1_0(\Omega)}\le 1}\|\mathcal{A}_{\e,k} y\|_{H^{-1}(\Omega)} \\&=\sup_{\|y\|_{H^1_0(\Omega)}\le 1}
\sup_{\|v\|_{H^1_0(\Omega)}\le 1}\left<\mathcal{A}_{\e,k} y,v\right>_{H^{-1}(\Omega);H_0^1(\Omega)}\\
 &= \sup_{\|y\|_{H^1_0(\Omega)}\le 1}
\sup_{\|v\|_{H^1_0(\Omega)}\le 1} \int_\Omega \left[\e+\mathcal{F}_k\big(|A^\frac{1}{2}\nabla y|^2\big)\right]^{\frac{p-2}{2}}\left(\nabla v,A\nabla y\right)_{\mathbb{R}^N}
dx\\
&\le \frac{\|\xi_2\|^2_{L^\infty(\Omega)}}{\left(\e+k^2+1\right)^{\frac{2-p}{2}}}\sup_{\|y\|_{H^1_0(\Omega)}\le 1}
\sup_{\|v\|_{H^1_0(\Omega)}\le 1} \|y\|_{H^1_0(\Omega)}\|v\|_{H^1_0(\Omega)}=C_{\e,k}.
\end{align*}

\textit{Strict monotonicity}
We make use of the following algebraic inequality, which is proved in \cite[Proposition 4.4]{KuMa151}:
\begin{equation*}
\label{3a.3.A.1}
\left(\big(\e+\mathcal{F}_k(|a|^2)\big)^{\frac{p-2}{2}}a-
\big(\e+\mathcal{F}_k(|b|^2)\big)^{\frac{p-2}{2}}b,a-b\right)_{\mathbb{R}^N} \ge \e^{\frac{p-2}{2}} |a-b|^2,\,a,b\in \mathbb{R}^N.
\end{equation*}
With this, having put $a:=A^\frac{1}{2}\nabla y$, $b:=A^\frac{1}{2}\nabla v$  we obtain
\begin{align*}
\big< \mathcal{A}_{\e,k}(A,y)-\mathcal{A}_{\e,k}(A,v), y-v\big>_{H^{-1}(\Omega);H^1_0(\Omega)}
&\ge \e^{\frac{p-2}{2}} \int_\Omega |A^\frac{1}{2}\nabla y-A^\frac{1}{2}\nabla v|^2  dx \\
&\ge
\alpha^2 \e^{\frac{p-2}{2}} \|y-v\|^2_{H^1_0(\Omega)}\ge 0.
\end{align*}
Since the relation
$
\big< \mathcal{A}_{\e,k}(A,y)-\mathcal{A}_{\e,k}(A,v), y-v\big>_{H^{-1}(\Omega);H^1_0(\Omega)}=0
$
implies $y=v$, it follows that the strict monotonicity property \eqref{1.4}--\eqref{1.4a} holds true for each $A\in \mathfrak{A}_{ad}$, $k\in \mathbb{N}$, and $\e>0$.

\textit{Coercivity.}
The coercivity property obviously follows from the estimate
\begin{equation}\label{3.4.new}
\big<\mathcal{A}_{\e,k}(A,y),y\big>_{H^{-1}(\Omega);H^1_0(\Omega)}
\ge\alpha^2\e^{\frac{p-2}{2}}\|y\|^2_{H_0^1(\Omega)}.
\end{equation}

\textit{Semi-continuity.}
In order to get the equality
\[
\lim_{t\to 0} \langle \mathcal{A}_{\e,k}(A,y+tw),v\rangle_{H^{-1}(\Omega);H^1_0(\Omega)}= \langle\mathcal{A}_{\e,k}(A,y),v\rangle_{H^{-1}(\Omega);H^1_0(\Omega)},
\]
it is enough to observe that
\begin{equation*}
(\e+\mathcal{F}_k(|A^\frac{1}{2}(\nabla y+t\nabla w)|^2))^{\frac{p-2}{2}}A\left(\nabla y+t\nabla w\right)\rightarrow (\e+\mathcal{F}_k(|A^\frac{1}{2}\nabla y|^2))^{\frac{p-2}{2}}A\nabla y,
\end{equation*}
as $t\to 0$ almost everywhere in $\Omega$, and make use of Lebesgue's dominated convergence theorem.

\textbf{Proof of Proposition \ref{prop 3.4}}

Similarly to the proofs of Proposition \ref{Prop 3a.2}, the boundedness, strict monotonicity, and radial continuity of $F_{\e,k}$ can be shown. It remains to prove the compactness property. Let $y_n\rightharpoonup y_0$  in $H_0^1(\Omega)$. Hence, $y_n\to y_0$ strongly in $L^2(\Omega)$ and, up to a subsequence, $y_n\to y_0$ a.e. in $\Omega$. We must show that $F_{\e,k}(y_n,z)\to F_{\e,k}(y_0,z)$ strongly in $L^2(\Omega)$, i.e.
\begin{multline}\label{3.51}
\int_\Omega|F_{\e,k}(y_n,z)- F_{\e,k}(y_0,z)|^2\,dx\\=\int_\Omega |(\e+\mathcal{F}_k(|y_n|^2))^{\frac{p-2}{2}}y_n-(\e+\mathcal{F}_k(|y_0|^2))^{\frac{p-2}{2}}y_0|^2dx\to 0\,\mbox{as} \,n\to \infty.
\end{multline}
Obviously, $|(\e+\mathcal{F}_k(|y_n|^2))^{\frac{p-2}{2}}y_n-(\e+\mathcal{F}_k(|y_0|^2))^{\frac{p-2}{2}}y_0|^2\to 0$ a.e. in $\Omega$. Also, the following estimate implies the equi-integrability property of this function
\begin{multline*}
\int_\Omega |(\e+\mathcal{F}_k(|y_n|^2))^{\frac{p-2}{2}}y_n-(\e+\mathcal{F}_k(|y_0|^2))^{\frac{p-2}{2}}y_0|^2\,dx\\
\le 2 (\e+k^2+1)^{p-2}\int_\Omega (|y_n|^2+|y_0|^2)\,dx\le C_{\e,k},\;\forall\, n\in\mathbb{N}.
\end{multline*}
Therefore, due to Lebesgue's Theorem \ref{Th_1.9},
$$|(\e+\mathcal{F}_k(|y_n|^2))^{\frac{p-2}{2}}y_n-(\e+\mathcal{F}_k(|y_0|^2))^{\frac{p-2}{2}}y_0|^2\to 0\mbox{ strongly in }L^1(\Omega),$$
 i.e. \eqref{3.51} holds true.

\textbf{Proof of Proposition \ref{prop 6.1}.}

Let us fix an arbitrary element $y$ of $H^1_0(\Omega)$. We associate with this element the set $\Omega_k(A,y)$, where $\Omega_k(A,y):=\left\{x\in\Omega\ :\ |A^\frac{1}{2}\nabla y(x)|>\sqrt{k^2+1}\right\}$. Then
\begin{align}
\label{3a.6.2.1}
\int_\Omega g y\,dx &= \|g\|_{L^2(\Omega)}\|y\|_{L^2(\Omega)}\stackrel{\text{by Friedrich's inequality}}{\le} C_\Omega \|g\|_{L^2(\Omega)}\|\nabla y\|_{L^2(\Omega)^N}\\\notag
&\le C_\Omega\|g\|_{L^2(\Omega)}\big[\|\nabla y\|_{L^2(\Omega\setminus\Omega_k(A,y))^N}+ \|\nabla y\|_{L^2(\Omega_k(A,y))^N}\big].
\end{align}
Since $\mathcal{F}_k(|A^\frac{1}{2}\nabla y|^2)=|A^\frac{1}{2}\nabla y|^2$ a.e. in $\Omega\setminus \Omega_k(A,y)$, and
$k^2\le \mathcal{F}_k(|A^\frac{1}{2}\nabla y|^2)\le k^2+ 1$ a.e. in $\Omega_k(A,y)$ $\forall\,k\in \mathbb{N}$,
we get
\begin{align*}
&\|\nabla y\|_{L^2(\Omega\setminus\Omega_k(A,y))^N}\le\alpha^{-1}\left(\int_{\Omega\setminus\Omega_k(A,y)}|A^\frac{1}{2}\nabla y|^2\,dx\right)^\frac{1}{2}\\
&\le
\alpha^{-1} |\Omega\setminus\Omega_k(A,y)|^{\frac{p-2}{2p}}\left(\int_{\Omega\setminus\Omega_k(A,y)} |A^\frac{1}{2}\nabla y|^p\,dx\right)^{\frac{1}{p}}\\
&\le
\alpha^{-1}|\Omega|^{\frac{p-2}{2p}} \left(\int_{\Omega\setminus\Omega_k(A,y)} (\e+|A^\frac{1}{2}\nabla y|^2)^{\frac{p-2}{2}}|A^\frac{1}{2}\nabla y|^2\,dx\right)^{\frac{1}{p}}\\
&=\alpha^{-1} |\Omega|^{\frac{p-2}{2p}} \left(\int_{\Omega\setminus\Omega_k(A,y)} (\e+\mathcal{F}_k(|A^\frac{1}{2}\nabla y|^2))^{\frac{p-2}{2}}|A^\frac{1}{2}\nabla y|^2\,dx\right)^{\frac{1}{p}}
\\ & \le \alpha^{-1}|\Omega|^{\frac{p-2}{2p}}\|y\|_{A,\e,k},
\end{align*}
and
\begin{multline*}
\|\nabla y\|_{L^2(\Omega_k(A,y))^N}\le \alpha^{-1}\left(\int_{\Omega_k(A,y)}|A^\frac{1}{2}\nabla y|^2\,dx\right)^{\frac{1}{2}}\\
\le \ds\frac{\alpha^{-1}}{k^\frac{p-2}{2}}\left(\int_{\Omega_k(A,y)} (\e+\mathcal{F}_k(|A^\frac{1}{2}\nabla y|^2))^{\frac{p-2}{2}}|A^\frac{1}{2}\nabla y|^2\,dx\right)^{\frac{1}{2}}\le \ds\frac{\alpha^{-1}}{k^\frac{p-2}{2}}\|y\|^{\frac{p}{2}}_{A,\e,k}.
\end{multline*}
As a result, inequality \eqref{3a.6.2.1} finally implies the desired estimate.
The proof is complete.

\textbf{Proof of Proposition \ref{prop 4.1}}

Due to strong convergence $\left(\e_n+\mathcal{F}_{k_n}(y_0^2)\right)^\frac{p-2}{2}y_0\rightarrow |y_0|^{p-2}y_0$ in $L^q(\Omega)$ and relations
\begin{multline}
\left(\e_n+\mathcal{F}_{k_n}(y_n^2)\right)^\frac{p-2}{2}y_n-\left(\e_n+\mathcal{F}_{k_n}(y_0^2)\right)^\frac{p-2}{2}y_0
=\left(\e_n+\mathcal{F}_{k_n}(y_n^2)\right)^\frac{p-2}{2}(y_n-y_0)\\+\left(\left(\e_n+\mathcal{F}_{k_n}(y_n^2)\right)^\frac{p-2}{2}-
\left(\e_n+\mathcal{F}_{k_n}(y_0^2)\right)^\frac{p-2}{2}\right)y_0=I_1^k+I_2^k,
\end{multline}
it is enough to prove that $I_i^k\to 0$ strongly in $L^q(\Omega)$ as $k\to \infty$ for $i=1,2$.

\textit{Step 1.} To prove  $\|I_1\|^q_{L^q}=\int_\Omega \left|\left(\e_n+\mathcal{F}_{k_n}(y_n^2)\right)^\frac{p-2}{2}\right|^\frac{p}{p-1}|y_n-y_0|^\frac{p}{p-1}\,dx\to 0$ we use Lemma \ref{Gikov}.
The initial suppositions imply that sequence $\{\varphi_n=|y_n-y_0|^\frac{p}{p-1}\}_{n\in\mathbb{N}}$ is bounded in $L^1(\Omega)$ and converges to 0 almost everywhere in $\Omega$. On this step we are left to prove only the equi-integrability property of the sequence
$\left|\e_n+\mathcal{F}_{k_n}(y_n^2)\right|^\frac{p(p-2)}{2(p-1)}$.  Let us notice, that $\ds\frac{p(p-2)}{2(p-1)}< \ds\frac{p-1}{2}<\ds\frac{p}{2}$ and, using H\"{o}lder inequality with exponents $1/s+1/s'=1$, where $s=\ds\frac{p}{2}/\ds\frac{p(p-2)}{2(p-1)}=\ds\frac{p-1}{p-2}$, $s'=p-1$, we have
\begin{multline}\label{6.6}
\int_\Omega\left|\e_n+\mathcal{F}_{k_n}(y_n^2)\right|^\frac{p(p-2)}{2(p-1)}\,dx
 \le\left(\int_\Omega \left|\e_n+\mathcal{F}_{k_n}(y_n^2)\right|^\frac{p}{2}dx\right)^\frac{p-2}{p-1}|\Omega|^\frac{1}{p-1}\\
 \le\left(\int_\Omega \left|\e_n+\mathcal{F}_{k_n}(y_n^2)\right|^\frac{p-2}{2}(\e_n+y_n^2)dx\right)^\frac{p-2}{p-1}|\Omega|^\frac{1}{p-1}\\
 =\left(\int_\Omega (J_1+J_2)\,dx\right)^\frac{p-2}{p-1}|\Omega|^\frac{1}{p-1}\le C|\Omega|^\frac{1}{p-1}.
\end{multline}
 Indeed, $\int_\Omega J_1\,dx =\e_n\int_\Omega \left|\e_n+\mathcal{F}_{k_n}(y_n^2)\right|^\frac{p-2}{2}\,dx\to 0$, because $\e_n\to 0$ and within a subsequence, still denoted by the same index, $\left|\e_n+\mathcal{F}_{k_n}(y_n^2)\right|^\frac{p-2}{2}\to |y_0|^{p-2}$ a.e. in $\Omega$. As for the second integral, we have
$$\int_\Omega J_2\,dx\le \|\left(\e_n+\mathcal{F}_{k_n}(y_n^2)\right)^\frac{p-2}{2}y_n\|_{L^2(\Omega)}\|y_n\|_{L^2(\Omega)}\le C\sup_{n\in \mathbb{N}}\|y_n\|_{L^2(\Omega)}.$$
\textit{Step 2.} Here we prove that $I_2^q\to 0$ strongly in $L^1(\Omega)$. Indeed, within a subsequence,
 $(\e_n+\mathcal{F}_{k_n}(y_n^2))^\frac{p-2}{2}- (\e_n+\mathcal{F}_{k_n}(y_0^2))^\frac{p-2}{2}\to 0$ a.e. in $\Omega$ and closely following the arguments of the previous step it can be shown that
 \begin{multline*}
 \int_\Omega\left|(\e_n+\mathcal{F}_{k_n}(y_n^2))^\frac{p-2}{2}-(\e_n+\mathcal{F}_{k_n}(y_0^2))^\frac{p-2}{2}\right|^\frac{p}{p-1}\,dx\\ \le  \left(\int_\Omega|\e_n+\mathcal{F}_{k_n}(y_n^2)|^\frac{p(p-2)}{2(p-1)}\,dx+\int_\Omega|\e_n+\mathcal{F}_{k_n}(y_0^2)|^\frac{(p-2)p}{2(p-1)}\,dx\right)^\frac{p}{p-1}\le c.
 \end{multline*}
It remains to apply Lemma \ref{Gikov}.

\textit{Acknowledgements.} The research was partially supported by Grant of the President of Ukraine GP/F61/017 and Grant of NAS of Ukraine 2284/15.

\end{document}